\documentclass[11pt,reqno]{amsproc}
\usepackage[margin=1in]{geometry}
\usepackage{amsmath, amsthm, amssymb,enumitem}
\usepackage{hyperref}
\hypersetup{colorlinks=true, pdfstartview=FitV, linkcolor=BrickRed,citecolor=BrickRed, urlcolor=BrickRed}
\usepackage[abbrev,lite,nobysame]{amsrefs}
\usepackage{times}
\usepackage[usenames,dvipsnames]{color}
\usepackage{mathtools}

\definecolor{labelkey}{rgb}{0,0,1}
\usepackage{dsfont}
\allowdisplaybreaks

\newtheorem{proposition}{Proposition}[section]
\newtheorem{theorem}[proposition]{Theorem}
\newtheorem{corollary}[proposition]{Corollary}
\newtheorem{lemma}[proposition]{Lemma}
\newtheorem*{bootstrap*}{Bootstrap Step}
\theoremstyle{definition}
\newtheorem{definition}[proposition]{Definition}
\newtheorem{remark}[proposition]{Remark}
\newtheorem{example}[proposition]{Example}
\newtheorem{assumption}[proposition]{Assumption}
\numberwithin{equation}{section}

\newcommand\e{{\rm e}}
\newcommand\dd{\,{\rm d}}
\newcommand\ddt{{\frac{\dd}{\dd t}}}
\def\Re{{\rm Re}}
\def\Im{{\rm Im}}
\def\l {\langle}
\def\r {\rangle}
\newcommand\de{{\partial}}

\newcommand{\EE}{\mathcal{E}}
\newcommand{\OO}{\mathcal{O}}
\newcommand{\dist}{\mathrm{dist}}
\newcommand{\sign}{\mathrm{sign}}
\newcommand{\supp}{\mathrm{supp}}
\newcommand{\1}{\mathbf{1}}
\newcommand{\Pe}{\mathrm{Pe}}

\newcommand{\NN}{\mathbb{N}}

\newcommand{\TT}{\mathbb{T}}
\newcommand{\RR}{\mathbb{R}}
\newcommand{\CC}{\mathbb{C}}
\renewcommand{\SS}{\mathbb{S}}

\newcommand\cI{{\mathcal I}}
\newcommand\cN{{\mathcal N}}
\newcommand\cV{{\mathcal V}}

\title[Enhanced dissipation and Taylor dispersion]{Enhanced dissipation 
and Taylor dispersion \\ in higher-dimensional parallel shear flows}

\author[M.\ Coti Zelati]{Michele Coti Zelati}
\address{Department of Mathematics, Imperial College London, London, SW7 2AZ, UK}
\email{m.coti-zelati@imperial.ac.uk}
\author[Th.\ Gallay]{Thierry Gallay}
\address{Institut Fourier, Universit\'e Grenoble Alpes, CNRS, 38610 Gi\`eres, France}
\email{Thierry.Gallay@univ-grenoble-alpes.fr}

\subjclass[2020]{35Q35, 35H10, 47B44, 76E05}

\keywords{Enhanced dissipation, Taylor dispersion, shear flows, resolvent
estimates, hypocoercivity.}

\begin{document}

\begin{abstract}
We consider the evolution of a passive scalar advected by a parallel 
shear flow in an infinite cylinder with bounded cross section, in 
arbitrary space dimension. The essential parameters of the problem are 
the molecular diffusivity $\nu$, which is assumed to be small, and 
the wave number $k$ in the streamwise direction, which can take 
arbitrary values. Under generic assumptions on the shear velocity $v$, 
we obtain optimal decay estimates for large times, both in the enhanced
dissipation regime $\nu \ll |k|$ and in the Taylor dispersion regime
$|k| \ll \nu$. Our results can be deduced from resolvent estimates using
a quantitative version of the Gearhart-Pr\"uss theorem, or can be established
more directly via the hypocoercivity method. Both approaches are explored
in the present example, and their relative efficiency is compared. 
\end{abstract}


\maketitle

\section{Introduction}\label{sec1}

The evolution of a passive scalar advected by a parallel shear flow and
undergoing molecular diffusion is an idealized problem which plays an important
role in hydrodynamic stability theory. This is perhaps the simplest model
demonstrating how advection by an incompressible flow, which has no dissipative
effect by itself, can strengthen the action of diffusion and lead to energy
dissipation at a much faster rate. The relative importance of advection and
diffusion is measured by the P\'eclet number, which is inversely proportional to
the molecular diffusion coefficient $\nu$. We are interested in the regime of
large P\'eclet numbers, where two different phenomena can occur depending on the
streamwise wavenumber $k$. If $\Pe^{-1} \ll |k|L$, where $L$ is a characteristic
length of the domain, the lifetime of the Fourier mode with wavenumber $k$ is
not proportional to $\Pe$, as in usual diffusion, but typically to $\Pe^{1/3}$
or $\Pe^{1/2}$ depending on the shear velocity. This phenomenon is usually
called {\em accelerated diffusion} or {\em enhanced dissipation}
\cites{BBG,CMV,RY}. In contrast the Fourier modes corresponding to $|k|L \ll
\Pe^{-1} \ll 1$ evolve diffusively, with an effective diffusion coefficient that
is proportional to $\Pe$ and therefore inversely proportional to the molecular
viscosity $\nu$. This effect, which is only observed in very long or infinite
cylinders, is called {\em Taylor dispersion} or sometimes {\em Taylor-Aris
  dispersion} \cites{Ta1,Ta2,Ar,YJ}.

From a mathematical point of view, numerous results describing the enhancement
of diffusion due to advection by a divergence-free vector field were obtained
both in the deterministic and in the stochastic setting, see
\cites{BBPS,CKRZ,CCZW,CZDE,CZD,FP1,FP2,MK,PS} and the references therein. For
the specific case of a parallel flow in a two-dimensional strip, the enhanced
dissipation effect for a passive scalar was thoroughly studied in
\cites{ABN,BCZ,Wei}, and the estimates derived on that model also play a crucial
role in the stability analysis of the shear flow as a stationary solution of the
Navier-Stokes equations, see \cites{BW,BVW,IMM,WZZ,GNRS,CZEW}. The corresponding
problem in higher-dimensional cylinders did not attract much attention so far,
except in particular examples such as the plane Couette flow \cite{BGM} and the
pipe Poiseuille flow \cite{CWZ}. On the other hand, a rigorous justification of
Taylor dispersion using self-similar variables and center manifold theory was
achieved in \cite{BCW}, see also \cite{MR2}. It is worth mentioning that,
although the enhanced dissipation and the Taylor dispersion have the same
physical origin, the mathematical techniques used in \cite{BCW} and \cite{BCZ}
are completely different, and rely on distinct assumptions on the shear
velocity.

In the present paper, we reopen the study of a passive scalar advected by a
parallel flow with a double goal: we aim at investigating the {\em
  higher-dimensional case}, which has received less attention so far, with an
approach that covers {\em in a unified way} the enhanced dissipation and the
Taylor dispersion regimes. To state our results, we introduce some notation. Let
$\Omega \subset \RR^d$ be a smooth bounded domain, and $v : \overline{\Omega}
\to \RR$ be a smooth function. We consider the evolution of a passive scalar in
the infinite cylinder $\Sigma = \RR \times \Omega \subset \RR^{d+1}$ under the
action of the shear velocity $u(x,y) = (v(y),0)^T$. The density $f(x,y,t)$ of
the passive scalar satisfies the advection-diffusion equation
\begin{equation}\label{eq:feq}
  \partial_t f(x,y,t) + v(y)\partial_x f(x,y,t) \,=\, \nu 
  \Delta f(x,y,t)\,, \quad (x,y) \in \Sigma\,, \quad t > 0\,,
\end{equation}
where $\nu > 0$ is the molecular diffusion coefficient and $\Delta = \partial_x^2 
+ \Delta_y$ denotes the Laplace operator acting on all variables $(x,y) \in \Sigma$.
We supplement \eqref{eq:feq} with homogeneous Neumann conditions at the boundary
$\partial \Sigma = \RR \times \partial\Omega$.  Applying a Galilean
transformation if needed, we can assume without loss of generality that the
shear velocity $v$ has zero average over $\Omega$. If $L$ denotes the diameter
of $\Omega$ and $U$ is the maximum of $|v|$ on $\overline{\Omega}$, the P\'eclet
number is defined as
\[
  \Pe \,=\, \frac{UL}{\nu}\,.
\]
We are interested in the long-time behavior of the solutions of \eqref{eq:feq} 
in the regime where $\Pe \gg 1$. It is convenient to introduce dimensionless
variables defined by 
\[
  \tilde x \,=\, \frac{x}{L}\,, \qquad  \tilde y \,=\, \frac{y}{L}\,, \qquad
  \tilde t \,=\, \frac{Ut}{L}\,, \qquad  \tilde v \,=\, \frac{v}{U}\,.  
\]
Dropping all tildes for notational simplicity, we arrive at the same equation
\eqref{eq:feq} where $L = U = 1$ and $\nu = \Pe^{-1}$ is now a dimensionless
parameter.  

Since equation~\eqref{eq:feq} is invariant under translations in the horizontal
direction, it is useful to consider the (partial) Fourier transform formally 
defined by
\begin{equation}\label{eq:fourier}
  \hat f(k,y,t) \,=\, \int_\RR f(x,y,t)\,\e^{-ikx}\dd x\,, 
  \quad k \in \RR\,, \quad y \in \Omega\,, \quad t > 0\,.
\end{equation}
This quantity satisfies the evolution equation
\begin{equation}\label{eq:feq2}
  \partial_t \hat f(k,y,t) + ikv(y) \hat f(k,y,t) \,=\, \nu 
  \bigl(-k^2 + \Delta_y\bigr) \hat f(k,y,t)\,, \quad y \in \Omega\,, 
  \quad t > 0\,,
\end{equation}
where the horizontal wavenumber $k \in \RR$ is now a parameter. The horizontal
diffusion $-\nu k^2$ in \eqref{eq:feq2} plays only a minor role in the regime 
we consider, and can be conveniently eliminated by the change of dependent
variables
\begin{equation}\label{eq:gdef}
  \hat f(k,y,t) \,=\, \e^{-\nu k^2 t} g(k,y,t)\,, \quad k \in \RR\,, 
  \quad y \in \Omega\,, \quad t > 0\,.
\end{equation}
This leads to the ``hypoelliptic'' evolution equation
\begin{equation}\label{eq:geq}
  \partial_t g(k,y,t) + ikv(y) g(k,y,t) \,=\, \nu \Delta_y g(k,y,t)\,, 
  \quad y \in \Omega\,, \quad t > 0\,,
\end{equation}
which is the starting point of our analysis. As already mentioned, we suppose that
$g$ satisfies the homogeneous Neumann conditions at the boundary, but our results
still hold, with a similar proof, if we assume instead that $g = 0$ on $\partial
\Omega$. We also suppose that the horizontal wave number $k$ is nonzero,
otherwise \eqref{eq:geq} reduces to the usual heat equation in $\Omega$.
It is not difficult to verify that, for all initial data $g_0 \in L^2(\Omega)$,
equation~\eqref{eq:geq} has a unique global solution $t \mapsto g(k,t) \in
C^0([0,+\infty),L^2(\Omega))$ such that $g(k,0) = g_0$. Here and in the sequel 
$g(k,t)$ is a shorthand notation for the map $y \mapsto g(k,y,t) \in L^2(\Omega)$.
Our goal is to estimate the decay rate of the solutions of \eqref{eq:geq}
as $t \to +\infty$.

We first consider the situation where the cross section $\Omega$ is 
one-dimensional. Our main result in this case can be stated as follows. 

\begin{theorem}\label{thm:main1}
Assume that $d = 1$, $\Omega = (0,L)$, and that $v : [0,L] \to \RR$ is a $C^m$ 
function, for some $m \in \NN^*$, whose derivatives up to order $m$ do not vanish 
simultaneously:
\begin{equation}\label{eq:vder}
  |v'(y)| + |v''(y)| + \dots + |v^{(m)}(y)| \,>\, 0\,, \quad 
  \hbox{for all } y \in [0,L]\,. 
\end{equation}
Then there exist positive constants $C_1, C_2$ such that, for all $\nu > 0$,
all $k \neq 0$, and all initial data $g_0 \in L^2(\Omega)$, the solution 
of \eqref{eq:geq} satisfies, for all $t \ge 0$, 
\begin{equation}\label{lamnuk}
  \|g(k,t)\|_{L^2(\Omega)} \,\le\, C_1\,\e^{-C_2 \lambda_{\nu,k}t}\,\|g_0\|_{L^2(\Omega)}\,, 
  \quad\hbox{where }~ \lambda_{\nu,k} \,=\,
  \begin{cases} \nu^{\frac{m}{m+2}} |k|^{\frac{2}{m+2}} & \hbox{if }~ 
  0 < \nu \le |k|\,,\\ \frac{k^2}{\nu} & \hbox{if }~ 0 < |k| \le \nu\,.
  \end{cases}
\end{equation}
\end{theorem}

The main novel feature of Theorem~\ref{thm:main1} is to exhibit a decay rate
$\lambda_{\nu,k}$ which undergoes a {\em continuous transition} from the
enhanced dissipation regime $\nu \ll |k|$ to the Taylor dispersion regime
$|k| \ll \nu$. In previous mathematical works, both situations were studied
using different methods, making the comparison more difficult.  When
$\nu \le |k|$, the expression \eqref{lamnuk} of $\lambda_{\nu,k}$ is certainly
not new: it was obtained in \cite{BCZ}, up to a logarithmic correction 
of purely technical origin, and it can also be deduced from the general 
criteria given in \cite{Wei}, with some additional work. This
instructive formula shows that the long-time behavior of the solutions of
\eqref{eq:geq} is determined, in the enhanced dissipation regime, by the
degree of {\em degeneracy of the critical points} of the shear function $v$. In the
most common cases,  the shear flows under consideration are either monotone ($m = 1$)
or have nondegenerate critical points ($m = 2$). Accordingly, the lifetime 
$1/\lambda_{\nu,k}$ of the Fourier mode indexed by $k$ is proportional to 
$\nu^{-1/3}$ or $\nu^{-1/2}$ when $\nu \ll 1$, and is therefore much shorter 
than the diffusive time scale $\nu^{-1}$. In the Taylor dispersion regime 
$|k| \le \nu$, the decay rate $\lambda_{\nu,k} = k^2/\nu$ has the same 
dependence upon the Fourier parameter $k$ as the purely diffusive rate 
$\nu k^2$, but $\lambda_{\nu,k} \gg \nu k^2$ when $\nu \ll 1$. So, in all 
cases, the expression \eqref{lamnuk} of $\lambda_{\nu,k}$ reveals the strong 
influence of the advection term on the solutions of \eqref{eq:geq} when the 
P\'eclet number $\nu^{-1}$ is sufficiently large. 

In the higher-dimensional case $d \ge 2$, the situation is similar and we still
expect that the decay rate of the solutions of \eqref{eq:geq} is determined,
when $\nu \ll |k|$, by the degree of degeneracy of the critical points of $v$.
This is more difficult to prove, however, because the behavior of a function
near its critical points can take more diverse forms in higher dimensions. To
limit the complexity, we assume here that $v$ is a {\em Morse function}, which
means that $v$ has only a finite number of critical points in $\Omega$, all of
which are nondegenerate. For simplicity, we also suppose that $v$ has no
critical point on the boundary $\partial\Omega$, although this additional
restriction could be dispensed with. Our second main result is:

\begin{theorem}\label{thm:main2}
Assume that $v : \overline{\Omega} \to \RR$ is a smooth Morse function with 
no critical point on the boundary $\partial\Omega$. There exist positive 
constants $C_1, C_2$ such that, for all $\nu > 0$, all $k \neq 0$, and all 
initial data $g_0 \in L^2(\Omega)$, the solution of \eqref{eq:geq} satisfies
estimate \eqref{lamnuk} for all $t \ge 0$, where $m = 1$ if $v$ has no 
critical point in $\Omega$ and $m = 2$ if $v$ has at least one critical
point in $\Omega$. 
\end{theorem}

Many classical examples, such as the plane Couette flow or the Poiseuille flow
in a cylindrical pipe, are covered by Theorem~\ref{thm:main2}, but of course one
can imagine more degenerate situations where $v$ is not a Morse function. In
Section~\ref{sec2} below, we formulate a general condition on the level sets of
$v$ (Assumption~\ref{Hypv}) which ensures that estimate \eqref{lamnuk} holds for
all $\nu > 0$ and all $k \neq 0$. We then prove that our assumption holds for 
a one-dimensional map satisfying \eqref{eq:vder} and for a Morse function in 
any dimension, but we also give other examples which do not fall into these
categories. One may conjecture that estimate \eqref{lamnuk} holds for some $m
\in \NN^*$ if $v \in C^m(\overline{\Omega})$ and if, for any critical point
$\bar y \in \overline{\Omega}$, there exists an integer $n \in \{1,\dots,m\}$
such that the $n$-th order differential ${\rm d}^n v(\bar y)$ is nondegenerate, 
but proving that using our techniques requires nontrivial additional work. 
In a different direction, we also believe that Theorem~\ref{thm:main2} remains 
true for Morse-Bott functions, whose critical points can form submanifolds of 
nonzero dimension. A thorough examination of these interesting questions
is left for a future work, but a modest discussion of possible extensions of
our results can be found in Section~\ref{ssec23} below.

At this point, it is important to observe that, although
Theorems~\ref{thm:main1} and \ref{thm:main2} treat the enhanced dissipation and
the Taylor dispersion regimes in a unified way, the minimal assumptions on $v$
that are needed to obtain the decay estimate \eqref{lamnuk} are very different
in both situations.  On the one hand, we expect that the expression of
$\lambda_{\nu,k}$ is optimal when $\nu \ll |k|$ if $v$ has indeed a critical
point where the $n$-th order differential vanishes for all $n < m$. In
particular, there is no enhanced dissipation effect at all if, for instance, $v$
is constant on a nonempty open subset of $\Omega$. In contrast, the following
result shows that estimate \eqref{lamnuk} holds in the Taylor dispersion regime
whenever the shear velocity $v$ is not identically constant.

\begin{theorem}\label{thm:main3}
If $v : \overline{\Omega} \to \RR$ is continuous and not identically
constant, there exist positive constants $C_1, C_2$ such that the decay 
estimate \eqref{lamnuk} holds in the Taylor dispersion regime $0 < |k| \le \nu$. 
\end{theorem}

The rest of this paper is organized as follows. In Section~\ref{sec2}, we derive
accurate resolvent estimates for the linear operator $H_{\nu,k} = -\nu \Delta_y 
+ ikv(y)$, which is (up to a sign) the generator of the evolution defined by 
\eqref{eq:geq}. We impose an abstract condition on the shear velocity which, 
in the enhanced dissipation regime, implies that all level sets of $v$ 
are ``$H^1$-thin'', in a sense that is made precise in Appendix~\ref{sec:A}. 
We then check the validity of our assumption in concrete situations, for 
one-dimensional maps satisfying \eqref{eq:vder} and for Morse functions in all 
space dimensions. Finally, the semigroup estimate \eqref{lamnuk} is obtained 
from the resolvent bounds using a quantitative version of the Gearhart-Pr\"uss 
theorem which was recently obtained in \cite{Wei}, see also \cite{HS2}. 
This concludes the proof of Theorems~\ref{thm:main1}--\ref{thm:main3}. 

In Section~\ref{sec3}, we give an alternative proof of Theorem~\ref{thm:main2}
using the hypocoercivity method of Villani~\cite{Vi}, which was already used in
\cite{BCZ} to establish the enhanced dissipation estimate when $d = 1$. This
second approach is, in some sense, more direct and more elementary, since it
relies on relatively straightforward energy estimates for the solutions of the
evolution equation \eqref{eq:geq} in $H^1(\Omega)$.  However, to avoid
problematic contributions from the boundary, we now have to impose that $g = 0$
on $\partial\Omega$, or alternatively that $\Omega = \TT^d$ (the $d$-dimensional
torus). Moreover, to reduce complexity, we restrict ourselves to the Morse
case where $m =1$ or $2$, which means that the shear velocity has a finite
number of critical points which are all nondegenerate. This implies in
particular that the lowest eigenvalue of the semi-classical Hamiltonian
$-\nu\Delta_y + |\nabla v|^2$ in $L^2(\Omega)$ is bounded from below by
$C \nu^{\frac{m-1}{m}}$ as $\nu \to 0$, and this information is used to derive
the differential inequalities that eventually lead to estimate \eqref{lamnuk},
up to a logarithmic correction in the enhanced dissipation regime when $m=2$. As
is shown in \cite{BCZ}, the restriction $m \le 2$ can be removed at the expense
of introducing an energy functional with variable coefficients, which is
constructed using a suitable partition of unity. In the Taylor dispersion
setting, we also recover Theorem \ref{thm:main3} under slightly more restrictive
assumptions on the velocity profile, see Remark~\ref{rem:TayHypo} below for a
precise statement.

The efficiency of the methods implemented in this paper is briefly compared 
in the final Section~\ref{sec4}. In Appendix~\ref{sec:A}, we introduce and 
study a specific notion of ``thinness'' for arbitrary subsets of $\RR^d$, 
which is closely related to our Assumption~\ref{Hypv}. Finally, a few technical 
estimates that are used in the proof of Theorem~\ref{thm:main2} are 
collected in Appendix~\ref{sec:B}. 

\medskip\noindent{\bf Acknowledgments.} Th.G. is partially supported by the grant 
SingFlows ANR-18-CE40-0027 of the French National Research Agency (ANR). 
M.C.Z. acknowledges funding from the Royal Society through a University Research Fellowship 
(URF\textbackslash R1\textbackslash 191492). We thank the anonymous referees for
insightful comments and helpful suggestions. 

\section{Resolvent estimates}\label{sec2}

This section is devoted to the proof of our main results using a first approach,
which relies on spectral theory and resolvent estimates. We recall that $\Omega
\subset \RR^d$ is a smooth bounded domain, and $v : \overline{\Omega} \to \RR$ a
smooth function. Given any $\nu > 0$ and any $k \neq 0$, the linear evolution 
equation \eqref{eq:geq} can we written in the abstract form
\begin{equation}\label{eq:Hdef}
  \partial_t g + H_{\nu,k} g \,=\,0 \,, \qquad \hbox{where}\qquad
  H_{\nu,k} \,=\, -\nu \Delta_y + ikv(y)\,.
\end{equation}
We consider $H_{\nu,k}$ as a linear operator in the Hilbert space $X = 
L^2(\Omega)$ with domain
\[
  D(H) \,=\, \bigl\{g \in H^2(\Omega)\,;\, \cN\cdot \nabla g = 0 
  \hbox{ on }\partial \Omega\bigr\}\,,
\]
where $\cN$ denotes the outward unit normal on the boundary $\partial\Omega$.
Being a bounded perturbation of the Neumann Laplacian $-\nu\Delta_y$ in $\Omega$,
the operator $H_{\nu,k}$ has compact resolvent, hence purely discrete spectrum
\cite{Ka}. Moreover, for all $g \in D(H)$, we have the identities
\begin{equation}\label{numrange}
  \Re\,\langle H_{\nu,k} g,g\rangle \,=\, \nu \|\nabla g\|^2 \ge 0\,,
  \qquad \hbox{and}\qquad \Im\,\langle H_{\nu,k} g,g\rangle \,=\, k
  \int_\Omega v(y)|g(y)|^2\dd y\,,
\end{equation}
where $\langle\cdot,\cdot\rangle$ denotes the scalar product in $X$ and $\|\cdot\|$ 
the corresponding norm. This implies that the numerical range of $H_{\nu,k}$ is 
included in the infinite strip
\[
  S_k \,=\, \bigl\{z \in \CC\,;\, \Re(z) \ge 0\,,~ |\Im(z)| \le |k|
  \|v\|_{L^\infty(\Omega)}\bigr\}\,.
\]
The eigenvalues of $H_{\nu,k}$ form a sequence $(\mu_n)_{n \in \NN}$ of complex
numbers which satisfy $\mu_n \in S_k$ for all $n \in \NN$ and
$\Re(\mu_n) \to +\infty$ as $n \to \infty$. Furthermore, if we assume that
$k \neq 0$ and $v$ is not identically constant, we deduce from the first
relation in \eqref{numrange} that $\Re(\mu_n) > 0$ for all $n \in \NN$, so that
the imaginary axis in the complex plane is included in the resolvent set of the
operator $H_{\nu,k}$. Our goal is to obtain accurate estimates on the resolvent
norm $\|(H_{\nu,k}-z)^{-1} \|$ for all $z \in i\RR$. In particular, we need a
precise lower bound on the pseudospectral abscissa
\begin{equation}\label{eq:Psidef}
  \Psi(\nu,k) \,:=\, \biggl(\sup_{z \in i\RR} \|(H_{\nu,k}-z)^{-1} \|\biggr)^{-1}\,,
\end{equation}
as a function of the parameters $\nu$ and $k$. This quantity was introduced
and studied in \cite{GGN} for a related problem. It is easy to verify that 
$\Re(\mu_n) \ge \Psi(\nu,k)$ for all $n \in \NN$. More importantly, a recent 
result due to Dongyi Wei \cite{Wei}, see also Helffer \& Sj\"ostrand \cite{HS2}, 
shows that the quantity $\Psi(\nu,k)$ entirely controls the decay rate of 
the semigroup generated by $-H_{\nu,k}$. Indeed, applying \cite{Wei}*{Theorem~1.3}, 
we obtain: 

\begin{proposition}\label{prop:Wei}
The operator $-H_{\nu,k}$ is the generator of a strongly continuous 
semigroup in $X$ which satisfies, for all $g \in X$ and all $t \ge 0$, 
\begin{equation}\label{eq:Wei}
  \|\e^{-tH_{\nu,k}} g\| \,\le\, \e^{\pi/2}\,\e^{-t\Psi(\nu,k)}\|g\|\,.
\end{equation}
\end{proposition}

So, to obtain the decay estimate \eqref{lamnuk} with $C_1 = \e^{\pi/2}$, all we 
need is to derive a lower bound of the form $\Psi(\nu,k) \ge C_2 \lambda_{\nu,k}$ 
for some positive constant $C_2$ that is independent of the parameters $\nu,k$. 

\medskip
To do that, we first exploit the assumption that $k \neq 0$ and write 
any $z \in i\RR$ in the form $z = ik\lambda$ with $\lambda \in \RR$, so 
that
\begin{equation}\label{eq:Hdef2}
  H_{\nu,k}-z \,=\, H_{\nu,k,\lambda} \,:=\, -\nu \Delta_y + ik\bigl(v(y)-
  \lambda\bigr)\,.
\end{equation}
It is apparent from \eqref{eq:Hdef2} that the estimates we can hope for 
depend on the properties of the level sets 
\[
  E_\lambda \,=\, \bigl\{y \in \Omega \,;\, v(y) = \lambda\bigr\}\,, 
  \quad \lambda \in \RR\,.
\]
For instance, if $E_\lambda$ has nonempty interior for some $\lambda \in \RR$,
then clearly $H_{\nu,k,\lambda}g = -\nu \Delta_y g$ for any function
$g \in D(H)$ that is supported in the interior of $E_\lambda$. As is easily
verified, this implies that $\Psi(\nu,k) \le C\nu$ for some positive constant
$C$, which means that there is no enhanced dissipation effect in such a case. 
So, to observe a nontrivial influence of the advection term in \eqref{eq:geq},
we must assume at least that all levels sets of the function $v$ are ``thin'' 
in an appropriate sense. 

To formulate our assumption precisely, we introduce the following notation. 
We give ourselves a positive integer $m \in \NN^*$, which will be related to 
the maximal degree of degeneracy of the critical points of $v$, as is explained 
in the introduction. For any $\lambda \in \RR$ and any $\delta > 0$, we then 
define the ``thickened level set''
\begin{equation}\label{eq:Edef}
  E_{\lambda,\delta}^m \,=\, \bigl\{y \in \Omega \,;\, |v(y) - \lambda| 
  < \delta^m\bigr\}\,,
\end{equation}
which is the union of the level sets $E_{\lambda'}$ for all $\lambda' \in 
(\lambda-\delta^m,\lambda+\delta^m)$. We also consider the $\delta$-neighborhood
\begin{equation}\label{eq:EEdef}
  \EE_{\lambda,\delta}^m \,=\, \bigl\{y \in \Omega \,;\, \dist(y,E_{\lambda,\delta}^m )
  < \delta\bigr\}\,,
\end{equation}
where ``$\dist$'' denotes the Euclidean distance in $\RR^d$. These definitions
are made so that the set $\EE_{\lambda,\delta}^m$ enjoys the following properties:

\begin{enumerate}[leftmargin=25pt]

\item[a)] If $y \in \Omega\setminus \EE_{\lambda,\delta}^m$, then $|v(y) - \lambda|
\ge \delta^m$ (this is clear from \eqref{eq:Edef} because $\EE_{\lambda,\delta}^m
\supset E_{\lambda,\delta}^m$).

\item[b)] A $\delta$-neighborhood of the level set $E_\lambda$ is included
in $\EE_{\lambda,\delta}^m$ (this follows from \eqref{eq:EEdef} since 
$E_{\lambda,\delta}^m \supset E_\lambda$).

\end{enumerate}

We can now formulate our general assumption on the function $v : \Omega \to \RR$.

\begin{assumption}\label{Hypv}
There exist a positive integer $m \in \NN^*$ and positive real constants 
$C_0, \delta_0$ such that, for all $\lambda \in \RR$ and all $\delta \in 
(0,\delta_0]$, the following inequality holds for all $g \in H^1(\Omega)$:
\begin{equation}\label{EEthin}
  \int_{\EE_{\lambda,\delta}^m} |g(y)|^2 \dd y \,\le\, \frac12 
  \int_\Omega  |g(y)|^2 \dd y + C_0 \delta^2 \int_\Omega  |\nabla g(y)|^2 \dd y\,. 
\end{equation}
\end{assumption}

\begin{remark}\label{rem:kappa}
The factor $1/2$ in \eqref{EEthin} can be replaced by any fixed real number 
$\kappa \in (0,1)$ without altering the definition, see Lemma~\ref{lem:kappa}
below for a similar statement. To avoid introducing yet another parameter, 
we stick to Assumption~\ref{Hypv} as it is stated, but it is useful to keep 
the general case in mind.   
\end{remark}

It is not easy to characterize precisely the functions $v$ that satisfy
Assumption~\ref{Hypv}, but the following observations can be made. First, we
emphasize that inequality \eqref{EEthin} must hold for {\em all sufficiently
  small} $\delta > 0$; having it satisfied for just {\em one} small $\delta > 0$
is infinitely less restrictive, as is shown in Lemma~\ref{lem:fixed} below.
Next, since $E_\lambda \subset E_{\lambda,\delta}^m$, Assumption~\ref{Hypv}
implies in particular that the level sets $E_\lambda$ are ``$H^1$-thin'' according
to the definition given in Appendix~\ref{sec:A}. As is shown there, any Lipschitz
graph or any submanifold of nonzero codimension is $H^1$-thin. As a consequence, 
if $\EE_{\lambda,\delta}^m$ is contained in a neighborhood of size $\OO(\delta)$ 
of a Lipschitz graph or a submanifold of nonzero codimension, then 
inequality \eqref{EEthin} holds. In contrast, if $\EE_{\lambda,\delta}^m$ 
contains a ball of radius $R(\delta)$ such that $R(\delta)/\delta \to +\infty$ 
as $\delta \to 0$, then \eqref{EEthin} fails. So Assumption~\ref{Hypv}
roughly means that, locally, the set $\EE_{\lambda,\delta}^m$ is no thicker than 
$\OO(\delta)$ in some direction. 

We can now explain the role played by the integer $m \in \NN^*$ in 
definitions \eqref{eq:Edef}, \eqref{eq:EEdef}. Assume for instance that 
$0 \in \Omega$ and that $v(y) = |y|^n$ near the origin, where $n \in \NN$ 
and $n \ge 2$. Then for $\delta > 0$ sufficiently small, the thickened 
level set $E_{0,\delta}^m$ contains the ball of radius $\delta^{m/n}$ 
centered at the origin, so that the $\delta$-neighborhood $\EE_{0,\delta}^m$
contains the ball of radius $\delta + \delta^{m/n}$. As we just saw, 
for inequality \eqref{EEthin} to hold, this radius must be $\OO(\delta)$ 
as $\delta \to 0$, which is the case if $m \ge n$. So a necessary condition
for Assumption~\ref{Hypv} to hold is that the integer $m$ be chosen 
sufficiently large, depending on the degree of degeneracy of the critical 
points of the function $v$. For instance, we can take $m = 1$ if $v$ has no 
critical points, and $m = 2$ if all critical points are nondegenerate, 
i.e. if $v$ is a Morse function. Note that a critical level set $E_\lambda$
may also contain noncritical points, in a neighborhood of which the
set $E_{\lambda,\delta}^m$ is as thin as $\OO(\delta^m)$ in the direction
normal to $E_\lambda$; this is the reason for which we consider the
$\delta$-neighborhood $\EE_{\lambda,\delta}^m$, which satisfies property 
b) above that will be needed in our argument.

The main result of this section is: 

\begin{proposition}\label{prop:resol}
Assume that the shear velocity $v : \Omega \to \RR$ satisfies 
Assumption~\ref{Hypv} for some positive integer $m \in \NN^*$. Then
there exists a constant $C > 0$ such that, for all $\nu > 0$ and all $k \neq 0$,
\begin{equation}\label{eq:resol}
  \Psi(\nu,k) \,\ge\, C\,\lambda_{\nu,k}\,,
\end{equation}
where $\Psi(\nu,k)$ is defined in \eqref{eq:Psidef} and $\lambda_{\nu,k}$ 
in \eqref{lamnuk}. 
\end{proposition}

\begin{proof}
Fix $\nu > 0$, $k \neq 0$, $\lambda \in \RR$, and $\delta \in (0,\delta_0)$, 
where $\delta_0 > 0$ is as in Assumption~\ref{Hypv}. For simplicity we denote 
$H = H_{\nu,k,\lambda}$, where $H_{\nu,k,\lambda}$ is defined in \eqref{eq:Hdef2}. 
We introduce the localization function 
\[
  \chi(y) \,=\, \phi\Bigl(\frac{1}{\delta}\,\sign\bigl(v(y)-\lambda\bigr)\,
  \dist\bigl(y,E_{\lambda,\delta}^m\bigr)\Bigr)\,, \qquad y \in \Omega\,,
\]
where $\phi : \RR \to [-1,1]$ is the unique odd function such that $\phi(t) 
= \min(t,1)$ for $t \ge 0$. We have the following three properties: 

\smallskip\noindent 
\quad i) $\chi$ is locally Lipschitz in $\Omega$ with $\|\nabla\chi\|_{L^\infty} \le 
1/\delta$; 

\smallskip\noindent 
\quad ii) $\chi(y)\bigl(v(y)-\lambda\bigr) \ge 0$ for all $y \in \Omega$;

\smallskip\noindent 
\quad iii) $\EE_{\lambda,\delta}^m = \{y \in \Omega\,;\,|\chi(y)| < 1\}$. 

\smallskip\noindent
These properties mean that $\chi$ is a Lipschitz regularization of the
discontinuous function $\sign(v-\lambda)$, such that the transitions between 
the values $-1$ and $+1$ occur within the region $\EE_{\lambda,\delta}^m$. 
To prove the Lipschitz continuity, we observe that $E_{\lambda,\delta}^m$ is an 
open neighborhood of the level set $E_\lambda$, so that $\dist(y,E_{\lambda,\delta}^m)$ 
vanishes near the points where $\sign(v(y)-\lambda)$ is discontinuous. 
The remaining properties are obvious by construction.

For any $g \in D(H)$ we have, as in \eqref{numrange},
\begin{equation}\label{eq:Hreal}
  \Re \langle Hg,g\rangle \,=\, \nu \|\nabla g\|^2\,, \qquad 
  \hbox{hence}\qquad \nu \|\nabla g\|^2 \,\le\, \|Hg\|\,\|g\|\,.
\end{equation}
Moreover, a direct calculation shows that $\Im \langle Hg,\chi g\rangle = 
\nu\,\Im\langle \nabla g,g \nabla\chi\rangle + k \langle \chi(v-\lambda)g,
g\rangle$. Since $|\chi| \le 1$ and $|\nabla\chi|\le 1/\delta$ we deduce 
that
\begin{equation}\label{eq:Himag}
  |k| \langle \chi(v-\lambda)g,g\rangle \,\le\, \|Hg\|\,\|g\| + 
  \frac{\nu}{\delta}\,\|\nabla g\|\,\|g\| \,\le\, \|Hg\|\,\|g\| + 
  \frac{\nu^{1/2}}{\delta}\,\|Hg\|^{1/2}\,\|g\|^{3/2}\,,
\end{equation}
where the second inequality follows from \eqref{eq:Hreal}. We now decompose
\begin{equation}\label{intdecomp}
  \|g\|^2 \,=\, \int_{\Omega\setminus \EE} |g(y)|^2\dd y \,+\, 
  \int_\EE |g(y)|^2\dd y\,, \qquad \hbox{where}\quad 
  \EE \,:=\, \EE_{\lambda,\delta}^m\,,
\end{equation}
and we estimate both terms separately using \eqref{eq:Hreal}, \eqref{eq:Himag}. 

\medskip\noindent
{\bf 1.} If $y \in \Omega\setminus \EE$, then $y \notin E_{\lambda,\delta}^m$, 
hence $|v(y)-\lambda| \ge \delta^m$ by definition. In view of properties ii), 
iii) above, we even have $\chi(y)\bigl(v(y)-\lambda\bigr) \ge \delta^m$, 
so that 
\begin{equation}\label{eq:int1}
\begin{split}
  \int_{\Omega\setminus \EE} |g(y)|^2\dd y \,&\le\, \frac{1}{\delta^m}
  \int_{\Omega\setminus \EE} \chi(y)\bigl(v(y)-\lambda\bigr) |g(y)|^2\dd y \,\le\, 
  \frac{1}{|k|\delta^m}\,|k| \langle \chi(v-\lambda)g,g\rangle \\ 
  \,&\le\, \frac{1}{|k|\delta^m}
  \Bigl(\|Hg\|\,\|g\| + \frac{\nu^{1/2}}{\delta}\,\|Hg\|^{1/2}\,\|g\|^{3/2}\Bigr) \\
  \,&\le\, \biggl(\frac{1}{|k|\delta^m} + \frac{\nu}{k^2 \delta^{2m+2}}\biggr)
  \|Hg\|\,\|g\| + \frac14\,\|g\|^2\,,
\end{split}
\end{equation}
where in the second line we used \eqref{eq:Himag} and in the third line
Young's inequality. 

\medskip\noindent
{\bf 2.} Since $\EE =  \EE_{\lambda,\delta}^m$, inequality \eqref{EEthin} gives
\begin{equation}\label{eq:int2}
  \int_\EE |g(y)|^2\dd y \,\le\, \frac12\,\|g\|^2 + C_0 \delta^2 \|\nabla g\|^2
  \,\le\, \frac12\,\|g\|^2 + \frac{C_0 \delta^2}{\nu}\,\|Hg\|\,\|g\|\,.
\end{equation}
  
\medskip\noindent
Combining \eqref{intdecomp}--\eqref{eq:int2}, we arrive at
\begin{equation}\label{eq:int3}
  \frac14 \,\|g\| \,\le\, \biggl(\frac{1}{|k|\delta^m} + \frac{\nu}{k^2 
  \delta^{2m+2}} +  \frac{C_0 \delta^2}{\nu}\biggr)\,\|Hg\|\,.
\end{equation}
We now choose
\begin{equation}\label{eq:deltachoice}
  \delta \,=\, \delta_0 \Bigl(\frac{\nu}{|k|}\Bigr)^{\frac{1}{m+2}} \quad 
  \hbox{if } ~\nu \le |k|\,, \qquad \hbox{and} \qquad  \delta \,=\, \delta_0 
  \quad \hbox{if } ~\nu \ge |k|\,.
\end{equation}
Then \eqref{eq:int3} shows that
\[
  \|Hg\| \,\ge\, C
  \begin{cases} \nu^{\frac{m}{m+2}} |k|^{\frac{2}{m+2}} \|g\| & 
   \hbox{if} \quad 0 < \nu \le |k|\,,\\[1mm]
   \frac{k^2}{\nu}\|g\| & \hbox{if} \quad 0 < |k| \le \nu\,,
  \end{cases}
\]
where the constant $C$ depends only on $C_0$, $\delta_0$, and $m$. 
In other words, we have $\|Hg\| \ge C \lambda_{\nu,k}\|g\|$ for all 
$g \in D(H)$, where $\lambda_{\nu,k}$ is as in \eqref{lamnuk}. Since 
$H = H_{\nu,k,\lambda}$ and 
\begin{equation}\label{eq:Psieq}
  \Psi(\nu,k) \,=\, \inf\Bigl\{\|H_{\nu,k,\lambda}g\|\,;\, \lambda \in \RR\,,
  ~g \in D(H)\,,~\|g\| = 1\Bigr\}\,,
\end{equation}
we obtain the desired inequality \eqref{eq:resol}. 
\end{proof}

It is clear that the decay estimate \eqref{lamnuk} follows immediately from
inequalities \eqref{eq:Wei} and \eqref{eq:resol}. So, to prove
Theorems~\ref{thm:main1}--\ref{thm:main3}, what remains to be done is verifying
the validity of Assumption~\ref{Hypv}. We first investigate under which
conditions a continuous function $v$ satisfies inequality \eqref{EEthin} for
some {\em fixed} $\delta > 0$.

\begin{lemma}\label{lem:fixed}
If $v : \overline{\Omega} \to \RR$ is continuous and not identically constant, 
then for any sufficiently small $\delta > 0$ there exist constants $C > 0$ and 
$\kappa \in (0,1)$ such that, for all $\lambda \in \RR$ and all $g \in H^1(\Omega)$, 
\begin{equation}\label{EEthin2}
  \int_{\EE_{\lambda,\delta}^1} |g(y)|^2 \dd y \,\le\, \kappa \int_\Omega  |g(y)|^2 \dd y 
  + C \int_\Omega  |\nabla g(y)|^2 \dd y\,. 
\end{equation}
\end{lemma}

\begin{proof}
Since $v$ is not constant, we can pick $y_1, y_2 \in \Omega$ such that 
$v(y_2) > v(y_1)$. We define
\[
  \lambda_0 \,=\, \frac{v(y_2) + v(y_1)}{2}\,, \qquad
  \gamma \,=\, \frac{v(y_2) - v(y_1)}{6}\,, 
\]
so that $v(y_1) = \lambda_0 -3\gamma$ and $v(y_2) = \lambda_0 +3\gamma$.  
We next choose $\delta > 0$ small enough so that\\[1mm]
\null\quad i) $\delta < \gamma$; \\[1mm]
\null\quad ii) $B(y_j,\delta) \subset \Omega$ for $j = 1,2$, where $B(y_j,\delta)$ 
is the open ball of radius $\delta$ centered at $y_j$;\\[1mm]
\null\quad iii) for all $y,\tilde y \in \Omega$ with $|y - \tilde y| < \delta$, 
one has $|v(y) - v(\tilde y)| < \gamma$.\\[1mm]
Property iii) holds because $v$ is continuous on the compact set 
$\overline{\Omega}$, hence uniformly continuous in $\Omega$. 

Given any $\lambda \in \RR$, we claim that
\begin{equation}\label{eq:auxinc}
  B(y_2,\delta) \,\subset\, \Omega\setminus\EE_{\lambda,\delta}^1 
  \quad \hbox{if}~ \lambda \le \lambda_0\,, \qquad \hbox{and}\qquad 
  B(y_1,\delta) \,\subset\, \Omega\setminus\EE_{\lambda,\delta}^1 
  \quad \hbox{if}~ \lambda \ge \lambda_0\,. 
\end{equation}
Let us prove the first assertion, the second one being similar. Assume thus 
that $\lambda \le \lambda_0$. By definition, if $y \in \EE_{\lambda,\delta}^1$, there 
exists $\tilde y \in E_{\lambda,\delta}^1$ such that $|y - \tilde y| < \delta$, hence 
by iii) and i) above
\[
  v(y) \,<\, v(\tilde y) + \gamma \,<\, \lambda + \delta + \gamma \,<\, 
  \lambda_0 + 2\gamma\,.
\]
On the other hand, for any $y \in B(y_2,\delta)$, one has $v(y) > v(y_2) - 
\gamma = \lambda_0 + 2\gamma$. It follows that $B(y_2,\delta) \cap 
\EE_{\lambda,\delta}^1 = \emptyset$, which is the first assertion in 
\eqref{eq:auxinc}. 

For any $\lambda \in \RR$, it follows from \eqref{eq:auxinc} that 
$|\EE_{\lambda,\delta}^1| \le |\Omega| - |B(\delta)|$, where $|B(\delta)|$ is the 
measure of a ball of radius $\delta$ in $\RR^d$. We now take $\rho > 0$ small 
enough so that
\begin{equation}\label{eq:kappadef}
  \kappa \,:=\, \bigl(1+\rho\bigr) \Bigl(1 - \frac{|B(\delta)|}{|\Omega|}
  \Bigr) \,<\, 1\,, \qquad \hbox{hence}\quad 
  |\EE_{\lambda,\delta}^1| \,\le\, \frac{\kappa\,|\Omega|}{1+\rho}\,.
\end{equation}
If $g \in H^1(\Omega)$, we decompose $g = \langle g\rangle + \tilde g$ where
$\langle g\rangle$ is the average of $g$ over $\Omega$. Using Young's 
inequality in the form $|g|^2 = \bigl|\langle g\rangle + \tilde g\bigr|^2 \le 
(1+\rho)|\langle g\rangle|^2 + \bigl(1 + 1/\rho\bigr) |\tilde g|^2$, we obtain
\begin{align*}
  \int_{\EE_{\lambda,\delta}^1} |g|^2 \dd y \,&\le\, (1+\rho)\int_{\EE_{\lambda,\delta}^1}
  |\langle g\rangle|^2 \dd y + \bigl(1 + 1/\rho\bigr)\int_{\EE_{\lambda,\delta}^1} 
  |\tilde g|^2 \dd y \\
  \,&\le\, (1+\rho)\frac{|\EE_{\lambda,\delta}^1|}{|\Omega|} \int_\Omega |\langle g
  \rangle|^2 \dd y + \bigl(1 + 1/\rho\bigr) \int_\Omega |\tilde g|^2 \dd y \\
  \,&\le\, \kappa \int_\Omega |g|^2 \dd y + \bigl(1 + 1/\rho\bigr) C_W^2 
  \int_\Omega |\nabla g|^2 \dd y\,,
\end{align*}
where in the last line we used \eqref{eq:kappadef}, the obvious bound 
$\|\langle g\rangle\| \le \|g\|$, and the Poincar\'e-Wirtinger inequality 
$\|\tilde g\| \le C_W \|\nabla\tilde g\|$. This gives the desired inequality
\eqref{EEthin2}.  
\end{proof}

\begin{proof}[Proof of Theorem~\ref{thm:main3}] 
In the Taylor dispersion regime where $|k| \le \nu$, the proof of 
Proposition~\ref{prop:resol} requires Assumption~\ref{Hypv} only for 
a fixed value of the parameter $\delta$, see \eqref{eq:deltachoice}. 
So, if $v : \overline{\Omega} \to \RR$ is continuous and not identically 
constant, we can use instead of \eqref{EEthin} inequality \eqref{EEthin2}, 
which is given by Lemma~\ref{lem:fixed}. The only (minor) difference is 
the factor $\kappa$ in the right-hand side of \eqref{EEthin2}, which 
may be larger than $1/2$, in which case one should modify \eqref{eq:int1}
so that the factor $1/4$ in the last member is replaced by $(1-\kappa)/2$. 
The rest of the proof is unchanged, and gives the desired inequality 
\eqref{eq:resol} when $|k| \le \nu$.   
\end{proof}

We next investigate the validity of inequality \eqref{EEthin} for {\em all sufficiently 
small} $\delta > 0$, which requires much stronger assumptions on the function $v$. 
We begin with the one-dimensional case, which is simpler. 

\subsection{The one-dimensional case}\label{ssec21} 
If $d = 1$, we can take $\Omega = (0,L)$, for some $L > 0$. Given a nonzero integer 
$m \in \NN$, we suppose that $v : [0,L] \to \RR$ is a function of class $C^m$ whose 
derivatives up to order $m$ do not vanish simultaneously:
\begin{equation}\label{eq:vder2}
  |v'(y)| + |v''(y)| + \dots + |v^{(m)}(y)| \,>\, 0\,, \quad 
  \hbox{for all } y \in [0,L]\,. 
\end{equation}
Our goal is to show that such a function satisfies Assumption~\ref{Hypv}, for 
the same value of $m$. 

We recall that $H^1(\Omega) \subset C^0(\overline{\Omega})$ and that the 
following inequality
\[
  \|g\|_{L^\infty}^2 \,\le\, \frac{1}{L}\,\|g\|_{L^2}^2 + 2 \|g\|_{L^2} \|g'\|_{L^2}
\]
holds for any $g \in H^1(\Omega)$. If $E \subset \Omega$ is any measurable set, 
we thus have
\begin{equation}\label{eq:Ebound}
  \int_E |g|^2 \dd y \,\le\, |E| \Bigl(\frac{1}{L}\,\|g\|_{L^2}^2 + 2 \|g\|_{L^2} 
  \|g'\|_{L^2}\Bigr) \,\le\, \Bigl(\frac14 + \frac{|E|}{L}\Bigr)\|g\|_{L^2}^2 + 
  4 |E|^2 \|g'\|_{L^2}^2\,,
\end{equation}
where in the second step we used Young's inequality. If $|E| \le \delta$ for some
$\delta \le L/4$, we see that inequality \eqref{eq:Ebound} is precisely of the 
form \eqref{EEthin}. This observation reveals that, in the one-dimensional case, 
it is sufficient to show that $|\EE_{\lambda,\delta}^m| = \OO(\delta)$ as $\delta \to 0$, 
uniformly with respect to $\lambda \in \RR$. 

We first estimate the length of the thickened level sets $E_{\lambda,\delta}^m$ 
defined in \eqref{eq:Edef}. 

\begin{lemma}\label{lem:E}
If $v \in C^m([0,L])$ satisfies \eqref{eq:vder2}, there exist positive constants
$C_0$, $\delta_0$ such that, for all $\lambda\in\RR$ and all $\delta \in (0,\delta_0)$,
one has $|E_{\lambda,\delta}^m| \le C_0\delta$. 
\end{lemma}

\begin{proof}
Since $E_{\lambda,\delta}^m = \emptyset$ when $\lambda \notin v([0,L])$ and 
$\delta$ is sufficiently small, we need only prove the result for $\lambda$ 
in a compact neighborhood of the range of $v$. Thus, by a finite covering 
argument, it is sufficient to prove that, for any $\lambda_0 \in \RR$, we have
the bound $|E_{\lambda,\delta}^m| \le C\delta$ for all $\lambda \in \RR$
sufficiently close to $\lambda_0$ and for all $\delta > 0$ small enough. 
From now on, we fix $\lambda_0 \in \RR$ and we consider the (extended)
level set
\[
  \overline{E}_{\lambda_0} \,=\, \Bigl\{y \in [0,L]\,;\, v(y) = \lambda_0\Bigr\}
  \,=\, \bigl\{y_1,\dots,y_N\bigr\}\,,
\]
which is a finite set since, by \eqref{eq:vder2}, the zeros of the function
$y \to v(y)-\lambda_0$ are isolated. If $\overline{E}_{\lambda_0} = \emptyset$, then
$E_\lambda = \emptyset$ when $\lambda$ is sufficiently close to $\lambda_0$,
hence also $E_{\lambda,\delta}^m = \emptyset$ when $\delta > 0$ is small enough.
So we need only consider the situation where $N := {\rm card}(\overline{E}_{\lambda_0}) \ge 1$.
In that case, a standard continuity argument shows that, for any $\epsilon > 0$,
there exists $\eta > 0$ such that, if $|\lambda - \lambda_0| < \eta$ and 
$0 < \delta < \eta$, any point $y \in E_{\lambda,\delta}^m$ satisfies $\dist(y,
\overline{E}_{\lambda_0}) < \epsilon$. In particular, if $\epsilon > 0$ is sufficiently small, 
the thickened level set $E_{\lambda,\delta}^m$ is included in the union of the 
pairwise disjoint intervals $(y_j-\epsilon,y_j+\epsilon)$, where $j = 1\dots N$. 
This means that it is enough to consider the intersection $E_{\lambda,\delta}^m \cap 
(y_j-\epsilon,y_j+\epsilon)$ for each $j \in \{1,\dots,N\}$, which reduces
the problem to the particular case $N = 1$. 

In the rest of the proof, we thus assume that $\overline{E}_{\lambda_0} = \{y_1\}$ for some 
$y_1 \in [0,L]$. According to \eqref{eq:vder2}, there exists $n \in \{1,\dots,m\}$
such that $v^{(j)}(y_1) = 0$ for $j = 1,\dots,n-1$ and $v^{(n)}(y_1) \neq 0$. 
We distinguish two cases according to the parity of $n$. 

\smallskip \noindent{\bf Case 1\thinspace:} $n$ is odd. For definiteness, 
we suppose that $y_1 \in (0,L)$ and $v^{(n)}(y_1) > 0$ (the other situations
can be treated in the same way). We introduce the function $w(x) = v(y_1+x) 
- \lambda_0$, which satisfies $w(0) = 0$. By assumption, there exist an 
open interval $\cI \subset \RR$ containing the origin and two positive 
constants $\kappa_1, \kappa_2$ such that
\[
  \kappa_1 x^{n-1} \,\le\, w'(x) \,\le\, \kappa_2 x^{n-1}\,, \qquad 
  \hbox{for all } x \in \cI\,.
\]
In particular, for any $\lambda$ in a small neighborhood $\cV$ of the origin, 
the equation $w(x) = \lambda$ has exactly one solution $x_\lambda$ in $\cI$, 
which satisfies $\lambda x_\lambda \ge 0$ and
\begin{equation}\label{eq:xlam}
  \frac{n|\lambda|}{\kappa_2} \,\le\, |x_\lambda|^n \,\le\, \frac{n|\lambda|}{\kappa_1}\,, 
  \qquad  \hbox{for all } \lambda \in \cV\,.
\end{equation}
If $w(x) \ge \lambda$ for some $x \in \cI$ and some $\lambda \in \cV$, one has
\[
  w(x) - \lambda \,=\, w(x) - w(x_\lambda) \,=\, \int_{x_\lambda}^x w'(y)\dd y
  \,\ge\, \frac{\kappa_1}{n}\bigl(x^n - x_\lambda^n\bigr)\,,
\]
and a similar estimate holds when $w(x) \le \lambda$. Reducing the neighborhood 
$\cV$ if necessary and choosing a sufficiently small $\delta_0 > 0$, we conclude 
that, if $\lambda \in \cV$ and $0 < \delta < \delta_0$, then
\begin{equation}\label{eq:Elamloc}
  x \in \tilde E_{\lambda,\delta}^n \,:=\, \bigl\{x \in \cI\,;\, 
  |w(x) - \lambda| \,<\, \delta^n\bigr\} \qquad \Longrightarrow\qquad
  |x^n - x_\lambda^n| \,<\, \frac{n \delta^n}{\kappa_1}\,.
\end{equation}

We now estimate the measure of the set $\tilde E_{\lambda,\delta}^n$
in \eqref{eq:Elamloc}. Assume for instance that $\lambda \ge 0$ (the
other case being similar). If $\lambda \ge \delta^n$, then
$x_\lambda \ge (n/\kappa_2)^{1/n}\delta$ by \eqref{eq:xlam}, hence by
\eqref{eq:Elamloc} any $x \in \tilde E_{\lambda,\delta}^n$ satisfies
\begin{equation}\label{eq:Elam1}
  |x-x_\lambda| \,<\, \frac{n \delta^n}{\kappa_1}\,\frac{1}{x^{n-1} + \dots + 
  x_\lambda^{n-1}} \,\le\, \Bigl(\frac{n}{\kappa_1}\Bigr) \Bigl(\frac{\kappa_2}{n}
 \Bigr)^{1-1/n}\delta\,.
\end{equation}
If $0 \le \lambda \le \delta^n$, then $x_\lambda \le (n/\kappa_1)^{1/n}\delta$
by \eqref{eq:xlam}, hence using \eqref{eq:Elamloc} again we find that
any $x \in \tilde E_{\lambda,\delta}^n$ satisfies
\[
  |x|^n \,\le\, |x^n - x_\lambda^n| + x_\lambda^n \,<\, \frac{2n}{\kappa_1}\,
  \delta^n\,.
\]
In all cases we obtain $|\tilde E_{\lambda,\delta}^n| \le C\delta$, where the 
constant depends only on $n,\kappa_1,\kappa_2$. Returning to the original 
function $v$, we conclude that, if $\lambda -\lambda_0 \in \cV$ and $0 < \delta 
< \delta_0$, we have the estimate
\begin{equation}\label{eq:Eldbd}
  |E_{\lambda,\delta}^m| \,\le\, |E_{\lambda,\delta}^n| \,=\, |\tilde E_{\lambda,\delta}^n|
  \,\le\, C(n,\kappa_1,\kappa_2) \delta\,.
\end{equation}

\smallskip \noindent{\bf Case 2\thinspace:} $n$ is even. We again assume that
$y_1 \in (0,L)$ and $v^{(n)}(y_1) > 0$. If $w(x) = v(y_1+x) - \lambda_0$, there
exists an open interval $\cI \subset \RR$ containing the origin such that
\[
  \kappa_1 x^n \,\le\, xw'(x) \,\le\, \kappa_2 x^n\,, \qquad 
  \hbox{for all } x \in \cI\,.
\]
If the parameter $\lambda$ is restricted to a small neighborhood $\cV$ of 
the origin, the equation $w(x) = \lambda$ has no solution in $\cI$ when 
$\lambda < 0$, and exactly two solutions $x_\lambda^\pm \in \cI$ when $\lambda > 0$, 
which satisfy 
\begin{equation}\label{eq:xlampm}
  x_\lambda^- < 0 < x_\lambda^+\,, \qquad \hbox{and}\qquad
  \frac{n\lambda}{\kappa_2} \,\le\, \bigl|x_\lambda^\pm\bigr|^n \,\le\, 
  \frac{n\lambda}{\kappa_1}\,. 
\end{equation}
Assume now that $\lambda \in \cV$ and $0 < \delta < \delta_0$ for some 
sufficiently small $\delta_0 > 0$. If $\lambda \ge \delta^n$ and $x$ 
belongs to the set $\tilde E_{\lambda,\delta}^n$ defined in \eqref{eq:Elamloc},
then taking $x_\lambda \in \{x_\lambda^-,x_\lambda^+\}$ such that $x x_\lambda \ge 0$ 
we have, as in \eqref{eq:Elamloc}, \eqref{eq:Elam1}, 
\[
  |x^n - x_\lambda^n| \,<\, \frac{n \delta^n}{\kappa_1}\,, \qquad \hbox{hence}
  \quad |x - x_\lambda| \,<\, \Bigl(\frac{n}{\kappa_1}\Bigr) \Bigl(\frac{\kappa_2}{n}
 \Bigr)^{1-1/n}\delta\,.
\]
If $\lambda \le \delta^n$ and $x \in \tilde E_{\lambda,\delta}^n$, then $w(x) < \lambda + 
\delta^n \le 2\delta^n$, and using the lower bound $w(x) \ge (\kappa_1/n)x^n$ we 
deduce that $|x| < (2n/\kappa_1)^{1/n}\delta$. In all cases we thus obtain 
$|\tilde E_{\lambda,\delta}^n| \le C\delta$, and we conclude as in \eqref{eq:Eldbd}. 
\end{proof}

The same estimate also holds for the $\delta$-neighborhoods of the 
thickened level sets: 

\begin{lemma}\label{lem:EE}
If $v \in C^m([0,L])$ satisfies \eqref{eq:vder2}, there exists positive constants
$C_1$, $\delta_1$ such that, for all $\lambda\in\RR$ and all $\delta \in (0,\delta_1)$,
one has $|\EE_{\lambda,\delta}^m| \le C_1\delta$. 
\end{lemma}

\begin{proof}
Again it is sufficient to estimate the measure of $\EE_{\lambda,\delta}^m$ for $\lambda$ 
close to some $\lambda_0 \in \RR$, and for $\delta > 0$ sufficiently small.  
The proof of Lemma~\ref{lem:E} shows that, under those assumptions, the thickened level 
set $E_{\lambda,\delta}^m$ is contained in the union of a finite number of intervals, the 
lengths of which are bounded by $C_0\delta$. Therefore, the $\delta$-neighborhood 
$\EE_{\lambda,\delta}^m$ is contained in a finite union of intervals of lengths 
$(C_0+2)\delta$, which gives the desired conclusion. 
\end{proof}

\begin{proof}[Proof of Theorem~\ref{thm:main1}] 
If $d = 1$ and $v$ satisfies \eqref{eq:vder}, Assumption~\ref{Hypv} holds 
as a consequence of Lemma~\ref{lem:EE} and estimate \eqref{eq:Ebound}. 
Thus, combining Propositions~\ref{prop:Wei} and \ref{prop:resol}, we obtain 
the desired conclusion. 
\end{proof}

\subsection{The case of Morse functions}\label{ssec22} 

Checking Assumption~\ref{Hypv} in the higher-dimensional case $d \ge 2$ is more
difficult, and we only consider here an important example. We assume that
$\Omega \subset \RR^d$ is a bounded domain with smooth boundary
$\partial \Omega$, and that $v : \overline{\Omega} \to \RR$ is a smooth {\em
  Morse function} with no critical point on the boundary. By this we first mean
that $v$ has only a finite number of critical points in $\Omega$, all of which
are nondegenerate. Moreover, by Whitney's extension theorem \cite{Hor1}, $v$ can
be extended to a smooth function on $\Omega_0 := \{y \in \RR^d\,;\, \dist(y,\Omega) 
<\epsilon_0\}$ for some $\epsilon_0 > 0$, and we assume that this extension 
(still denoted by $v$) has no critical point on $\partial\Omega$.

\begin{lemma}\label{lem:EEd}
If $v : \overline{\Omega} \to \RR$ is a smooth Morse function with no critical point 
on $\partial\Omega$, then Assumption~\ref{Hypv} holds with $m = 1$ if $v$ has 
no critical point in $\Omega$, and with $m = 2$ if $v$ has at least one critical 
point in $\Omega$. 
\end{lemma}

\begin{proof}
As in Lemma~\ref{lem:E}, it is sufficient to prove that, for all
$\lambda_0 \in \RR$, there exists a constant $C >0$ such that inequality
\eqref{EEthin} holds for all $\lambda$ sufficiently close to $\lambda_0$, all
$\delta > 0$ sufficiently small, and all $g \in H^1(\Omega)$. We thus fix
$\lambda_0 \in \RR$ and we denote by $y_1,\dots,y_N \in \Omega$ the critical 
points of $v$ that lie on the level set $E_{\lambda_0}$ (if there are none, we 
simply set $N = 0$). By the {\em Morse lemma} \cite{Mi}, for any $j \in \{1,\dots,N\}$ 
and any sufficiently small $\epsilon > 0$, there exist a neighborhood $V_j$ 
of $y_j$ in $\Omega$ and a smooth diffeomorphism $\phi_j : B(0,\epsilon) \to 
V_j$ such that $\phi_j(0) = y_j$ and 
\begin{equation}\label{vMorse}
  v(\phi_j(x)) = \lambda_0 + |x'|^2 - |x''|^2\,, \qquad \hbox{for all } 
  \,x = (x',x'') \in B(0,\epsilon) \,\subset\,\RR^d\,,
\end{equation}
where $x' \in \RR^{d_1}$, $x'' \in \RR^{d_2}$ with $d_1 + d_2 = d$. In other
words, after a smooth change of coordinates, we can assume that $v$ takes the
canonical form \eqref{vMorse} near the critical point $y_j$.

Similarly, in a neighborhood of any non-critical point $y \in 
\overline{E}_{\lambda_0} := \{y \in \overline{\Omega} \,;\, v(y) = \lambda_0\}$, 
we can transform $v$ into an affine function by a change of coordinates. 
The compact set $\overline{E}_{\lambda_0}$ being covered by these neighborhoods 
and by the sets $V_j$ for $j = 1,\dots,N$, we can extract a finite subcover. 
We can thus find $M$ points $y_{N+1},\dots,y_{N+M} \in \overline{E}_{\lambda_0}$ 
such that, for any $j \in \{N{+}1,\dots,N{+}M\}$, there exist a neighborhood 
$V_j$ of $y_j$ in $\Omega_0$ and a smooth diffeomorphism $\phi_j : 
B(0,\epsilon) \to V_j$ such that $\phi_j(0) = y_j$ and 
\begin{equation}\label{vMorse2}
  v(\phi_j(x)) = \lambda_0 + x_d \,, \qquad \hbox{for all } 
  \,x = (x_1,\dots,x_d) \in B(0,\epsilon)\,.
\end{equation}
Moreover we have 
\begin{equation}\label{cover}
  \overline{E}_{\lambda_0} \,\subset\, V \,:=\, V_1 \cup \dots \cup V_N \cup V_{N+1}
  \cup \dots \cup V_{N+M} \,\subset\, \Omega_0\,.
\end{equation}

Our next tool is a smooth partition of unity $(\chi_j)$ associated with the open 
cover \eqref{cover}. More precisely, there exists smooth functions $\chi_j 
: \RR^d \to \RR$ such that $K_j := \supp(\chi_j) \subset V_j$ for $j = 1,\dots,N+M$, 
and 
\begin{equation}\label{chipart}
  \sum_{j=1}^{N+M} \chi_j(y)^2 \,\le\, 1 \quad \hbox{for all }\, y \in \Omega_0\,,
  \qquad 
  \sum_{j=1}^{N+M} \chi_j(y)^2 \,=\, 1 \quad \hbox{for all }\, y \in E\,,
\end{equation}
for some open set $E \subset \RR^d$ with $\overline{E}_{\lambda_0} \subset E 
\subset \overline{E} \subset V$. For later use, we observe that $\EE_{\lambda,\delta}^m 
\subset E$ whenever $\lambda$ is sufficiently close to $\lambda_0$ and $\delta > 0$ is 
sufficiently small. 

Now, let $g \in H^1(\Omega)$. Since the boundary $\partial\Omega$ is smooth, 
we can extend $g$ to a function $\tilde g \in H^1(\Omega_0)$ which satisfies
\begin{equation}\label{gext}
  \|\tilde g\|_{L^2(\Omega_0)} \,\le\, 2 \|g\|_{L^2(\Omega)}\,, \qquad
  \|\tilde g\|_{H^1(\Omega_0)} \,\le\, C \|g\|_{H^1(\Omega)}\,, 
\end{equation}
for some constant $C > 0$ (independent of $g$). This extension will allow us 
to treat the boundary points $y_j \in \partial\Omega$ exactly as the interior 
points $y_j \in \Omega$. Since  $\EE_{\lambda,\delta}^m \subset E$ when $\lambda$ 
is sufficiently close to $\lambda_0$ and $\delta > 0$ is sufficiently small, 
we can use the partition of unity \eqref{chipart} to decompose
\begin{equation}\label{EEdecomp}
  \int_{\EE_{\lambda,\delta}^m} |g(y)|^2 \dd y \,=\, 
  \sum_{j=1}^{N+M}  \int_{\EE_{\lambda,\delta}^m} |g(y)|^2 \chi_j(y)^2\dd y 
  \,=\, \sum_{j=1}^{N+M}  \int_{\EE_{\lambda,\delta}^m} |g_j(y)|^2\dd y\,,
\end{equation}
where $g_j := \tilde g\,\chi_j$ is supported in $K_j \subset V_j$ for $j = 
1,\dots,N{+}M$. 

It remains to estimate the integrals in the right-hand side of \eqref{EEdecomp}, 
which can be written in the form
\begin{equation}\label{IJdef}
  I_j \,:=\, \int_{\EE_{\lambda,\delta}^m \cap K_j} |g_j(y)|^2\dd y \,=\,
  \int_{\phi_j^{-1}(\EE_{\lambda,\delta}^m \cap K_j)} |g_j(\phi_j(x))|^2 \,|J_{\phi_j}(x)|
  \dd x\,,
\end{equation}
where $J_{\phi_j}$ is the Jacobian determinant of the diffeomorphism $\phi_j$. 
We can find constants $L,\Lambda \ge 1$ such that 
\[
  L^{-1}\,|x-\tilde x| \,\le\, |\phi_j(x) - \phi_j(\tilde x)|  
  \,\le\, L |x-\tilde x|\,, \qquad \hbox{and}\qquad
  \Lambda^{-1} \,\le\, |J_{\phi_j}(x)| \,\le\, \Lambda\,,
\]
for all points $x,\tilde x \in B(0,\epsilon)$ and all integers $j \in 
\{1,\dots,N{+}M\}$. It is thus straightforward to verify that $\phi_j^{-1}
(\EE_{\lambda,\delta}^m \cap K_j) \subset \EE_{\lambda,L\delta,j}^m$ if $\delta > 0$ 
is small enough, where 
\[
  \EE_{\lambda,\delta,j}^m \,:=\, \bigl\{x \in B(0,\epsilon)\,;\, \dist(x,
  E_{\lambda,\delta,j}^m) < \delta \bigr\}\,, \quad 
  E_{\lambda,\delta,j}^m \,:=\, \bigl\{x \in B(0,\epsilon)\,;\, \big|
  v(\phi_j(x)) - \lambda\big| < \delta^m \bigr\}\,.
\]
Note that the sets $E_{\lambda,\delta,j}^m, \EE_{\lambda,\delta,j}^m$ are defined in terms 
of the canonical forms $v\circ\phi_j$ exactly as the sets \eqref{eq:Edef}, 
\eqref{eq:EEdef} were defined in terms of the original function $v$. We now
distinguish two cases:

\smallskip\noindent{\bf 1\thinspace:} {\em The critical points}. If 
$N \ge 1$ and $j \in \{1,\dots,N\}$, we necessarily have $m = 2$. In view 
of the canonical form \eqref{vMorse}, we apply Lemma~\ref{lem:Morse1} if 
$d_1 d_2 = 0$ (the case of a local extremum) or Lemma~\ref{lem:Morse2} 
if $d_1 d_2 > 0$ (the case of a saddle point). Denoting $h_j = g_j \circ \phi_j$, 
we thus obtain
\begin{equation}\label{IJest}
  I_j \,\le\, \Lambda \int_{\EE_{\lambda,L\delta,j}^m} |g_j(\phi_j(x))|^2 \dd x 
  \,\le\, C \delta \|h_j\| \|\nabla h_j\| \,\le\, \kappa 
  \|h_j\|^2 + C \delta^2 \|\nabla h_j\|^2\,,
\end{equation}
where the constant $\kappa > 0$ can be taken arbitrarily small. Returning
to the original variable $y$ we arrive at
\begin{equation}\label{hnabla}
\begin{split}
  \|h_j\|_{L^2}^2 \,&=\, \int_{B(0,\epsilon)} |g_j(\phi_j(x))|^2 \dd x \,\le\, 
  \Lambda\int_{V_j} |g_j(y)|^2 \dd y \,=\, \Lambda\int_{V_j} |\tilde g_j(y)|^2 
  \chi_j(y)^2 \dd y \\
  \|\nabla h_j\|_{L^2}^2 \,&\le\, C \int_{V_j} |\nabla g_j(y)|^2 \dd y \,\le\, 
  C \int_{V_j} \Bigl(|\nabla \tilde g(y)|^2 \chi_j(y)^2 + |\tilde g(y)|^2 
  |\nabla \chi_j(y)|^2\Bigr)\dd y\,.
\end{split} 
\end{equation}

\smallskip\noindent{\bf 2\thinspace:} {\em The non-critical points}. 
If $j \in \{N{+1},\dots,N{+}M\}$ and $m = 1$ or $2$, then due to the simple canonical 
form \eqref{vMorse2} we have the inclusions $\EE_{\lambda,L\delta,j}^m \subset 
\EE_{\lambda,L\delta,j}^1 \subset E_{\lambda,2L\delta,j}^1$. Therefore, applying 
Lemma~\ref{lem:lip} we obtain the same estimate \eqref{IJest} for the 
quantity $I_j$, and \eqref{hnabla} is unchanged too. Note that, if $y_j 
\in \partial\Omega$, the open set $V_j$ is not entirely contained in $\Omega$, 
so that we need to use the extension $\tilde g$ instead of $g$ in the 
right-hand side of \eqref{hnabla}. 

\smallskip
Summarizing, we deduce from \eqref{chipart}, \eqref{gext}, and 
\eqref{EEdecomp}--\eqref{hnabla} that
\begin{align*}
  \int_{\EE_{\lambda,\delta}^m} |g(y)|^2 \dd y \,=\, \sum_{j=1}^{N+M} I_j
  \,&\le\, \kappa \Lambda \int_{\Omega_0} |\tilde g(y)|^2 \dd y + C\delta^2 
  \int_{\Omega_0}\Bigl(|\nabla\tilde g(y)|^2 + |\tilde g(y)|^2\Bigr)\dd y \\
  \,&\le\, \frac{1}{2}\int_\Omega |g(y)|^2 \dd y + C\delta^2 \int_\Omega
  |\nabla g(y)|^2\dd y\,,
\end{align*}
provided $\kappa \Lambda < 1/8$ and $\delta > 0$ is sufficiently. This gives the desired 
estimate \eqref{EEthin}. 
\end{proof}

\begin{proof}[Proof of Theorem~\ref{thm:main2}] 
If $v$ is a Morse function, Assumption~\ref{Hypv} holds with $m = 1$ or $2$ 
in view of Lemma~\ref{lem:EEd}, and the desired conclusion follows from 
Propositions~\ref{prop:Wei} and \ref{prop:resol}. 
\end{proof}

\subsection{Additional examples}\label{ssec23} 

Lemmas~\ref{lem:EE} and \ref{lem:EEd} give general conditions that imply
the validity of Assumption~\ref{Hypv}, but our approach has a much broader 
scope and allows us potentially to treat many flows which do not fall into
these categories. We first consider two illustrative examples, and then
briefly discuss in which direction the assumptions of Theorem~\ref{thm:main2}
could be weakened. 

\begin{example}\label{ex:degen}
Assume that $\Omega = B(0,1) \subset \RR^d$ and that $v(y) = 1 - |y|^m$, 
for some $m \in \NN^*$. Then Assumption~\ref{Hypv} is verified for that value 
of $m$, so that estimate \eqref{lamnuk} holds for all $\nu > 0$ and all 
$k \neq 0$.
\end{example}

To see this, we observe that, if $\lambda \in \RR$ and $\delta > 0$, we
have $\EE_{\lambda,\delta}^m \subset \bigl\{y \in \Omega \,;\, R_1 \le |y| < R_2
\bigr\}$, where
\[
  R_1 \,=\, \bigl(1-\lambda-\delta^m\bigr)_+^{1/m} - \delta\,, \qquad
  R_2 \,=\, \bigl(1-\lambda+\delta^m\bigr)_+^{1/m} + \delta\,.
\]
An easy calculation shows that $R_2 - R_1 \le 4\delta$. So, if $g \in H^1(\Omega)$ 
and if $\tilde g \in H^1(\RR^d)$ is an extension of $g$ to $\Omega_0 = 
\RR^d$ satisfying \eqref{gext}, we can apply Lemma~\ref{lem:annulus} in 
Appendix~\ref{sec:B} to obtain the estimate
\[
  \int_{\EE_{\lambda,\delta}^m} |g(y)|^2\dd y \,\le\, 
  \int_{\{R_1 \le |y| < R_2\}} |\tilde g(y)|^2\dd y \,\le\, 
  8\delta \,\|\tilde g\|_{L^2} \,\|\nabla \tilde g\|_{L^2}\,,
\]
from which \eqref{EEthin} easily follows using \eqref{gext} and Young's 
inequality. Note that the value $m = 1$ is allowed, in which case the 
function $v$ is not smooth. In fact, we can allow $m$ to be an arbitrary 
positive real number in that example, but the constant $C_0$ in \eqref{EEthin}
diverges in the singular limit $m \to 0$.  

\begin{example}\label{ex:MorseBott}
Assume that $\Omega = B(0,1) \times B(0,1) \subset \RR^{d_1} \times \RR^{d_2}$ 
and that $v(y,z) = 1 - |y|^m$ for all $(y,z) \in \Omega$. If $d_1 \ge 1$ we have 
the same conclusions as in Example~\ref{ex:degen}. 
\end{example}

We can assume that $d_2 \ge 1$ too, otherwise we recover Example~\ref{ex:degen}. 
In the case $m = 2$, which is already interesting, the function $v$ is 
Morse-Bott: its critical points form a submanifold $S = \{(0,z)\,;\, 
|z| < 1\}$, and the second differential of $v$ is nondegenerate in the 
directions that are transverse to $S$. This example can be treated 
in the same way as the previous one. We have $\EE_{\lambda,\delta}^m \subset 
\bigl\{(y,z) \in \Omega \,;\, R_1 \le |y| < R_2\bigr\}$, with $R_1, R_2$ as 
above, so using Lemma~\ref{lem:annulus} and Fubini's theorem we obtain
\[
  \int_{\EE_{\lambda,\delta}^m} |g(y,z)|^2\dd y\dd z \,\le\, 
  \int_{\RR^{d_2}}\int_{\{R_1 \le |y| < R_2\}} |\tilde g(y,z)|^2\dd y\dd z \,\le\, 
  8\delta \,\|\tilde g\|_{L^2} \,\|\nabla \tilde g\|_{L^2}\,,
\]
and estimate \eqref{EEthin} follows by the same argument. Note that the domain 
$\Omega$ is not smooth in that example. 

\medskip
In a broader perspective, it is possible to prove the validity of
Assumption~\ref{Hypv} for smooth functions $v : \Omega \to \RR$ with degenerate
critical points under two conditions. The first one is that the degree of
degeneracy of the critical points be finite at least in some direction of the
space. Morse-Bott functions, for instance, may have critical level sets that
form smooth submanifolds of finite codimension, but in transverse directions
to the submanifold the critical points are nondegenerate, and this is sufficient
to obtain inequality \eqref{EEthin} with $m = 2$ using the same techniques as in
Lemma~\ref{lem:EEd} and Example~\ref{ex:MorseBott}. The details are postponed to
a future work.  On the other hand, if $v$ has an infinitely degenerate critical
point, such as in the classical example $v(y) = \exp(-1/|y|)$, it is rather
obvious that inequality \eqref{EEthin} fails for any choice of $m \in \NN^*$.

The second condition is more technical in nature: loosely speaking, there should
exist a convenient normal form describing the behavior of $v$ in a neighborhood
of any critical point. In the one-dimensional case, we can use the Taylor
expansion at the first nontrivial order, and for nondegenerate critical points
in higher dimensions the Morse lemma describes exactly the local behavior of the
function. To our knowledge, there is no convenient analogue of the Morse lemma for
degenerate critical points, and this is what really prevents us extending
Theorem~\ref{thm:main1} to higher dimensions, although simple cases such as
Example~\ref{ex:degen} can be treated by {\em ad hoc} arguments.

\subsection{A comment on optimality}\label{ssec24}

It is natural to ask whether our Assumption~\ref{Hypv} is necessary for
the result of Proposition~\ref{prop:resol} to hold. At present time we
cannot formulate any precise statement in this direction, but we believe
that hypothesis \eqref{EEthin} is not far from optimal. At least
in simple geometric situations, the failure of that inequality implies
properties of the (thickened) level sets of the function $v$ which
are incompatible with estimate \eqref{eq:resol} in the enhanced
dissipation regime.

To explain that, we consider the one-dimensional domain $\Omega = \RR$, which is
noncompact so that the results of Section~\ref{ssec21} do not apply. For
simplicity, we assume that $v : \RR \to \RR$ is a smooth, increasing and
globally Lipschitz function, for which Assumption~\ref{Hypv} with $m = 1$ does
not hold. This means that, for arbitrary small values of $\delta > 0$,
inequality \eqref{EEthin} fails for arbitrary large values of the constant $C_0$
and appropriate choices of the parameter $\lambda \in \RR$. Since $v$ is
increasing, the sets $E_{\lambda,\delta}^1$ and $\EE_{\lambda,\delta}^1$ are
just open intervals in the present case, and if we denote by $\ell$ the length
of $E_{\lambda,\delta}^1$ it is easily verified that the failure of
\eqref{EEthin} is equivalent to the property that $\ell/\delta \to +\infty$ as
$\delta \to 0$.

In what follows we assume for notational simplicity that $E_{\lambda,\delta}^1
= (-\ell/2,\ell/2)$. We give ourselves a smooth function $h \in C^2_c(\RR)$
which is not identically zero and satisfies $\supp(h) \subset (-1/2,1/2)$. 
We consider the function $g(y) := h(y/\ell)$, which is supported
in $E_{\lambda,\delta}^1$, and we compute $Hg$ where $H =  H_{\nu,k,\lambda}$ is
the operator defined in \eqref{eq:Hdef2}. The result is
\[
  \bigl(Hg\bigr)(y) \,=\, -\frac{\nu}{\ell^2}\,h''(y/\ell) + ik
  \bigl(v(y)-\lambda\bigr)\,h(y/\ell)\,, \quad y \in \RR\,.
\]
As $|v(y)-\lambda| < \delta$ when $y \in \supp(g)$, we easily deduce that
\begin{equation}\label{eq:Hgup}
  \|Hg\| \,\le\, C\Bigl(\frac{\nu}{\ell^2} + |k|\delta
  \Bigr) \|g\|\,,
\end{equation}
where the constant $C > 0$ does not depend on $\delta$, $\nu$, or $k$.

Now, if $\nu = |k| \ell^2 \delta$, which is compatible with the enhanced
dissipation regime, both terms inside the parenthesis are equal, and inequality
\eqref{eq:Hgup} becomes $\|Hg\| \le C\,\nu^{1/3}\,|k|^{2/3}\,\epsilon^{2/3}\|g\|$,
where $\epsilon := \delta/\ell$.  In view of \eqref{eq:Psieq} we deduce that
$\Psi(\nu,k) \le C \nu^{1/3}\,|k|^{2/3} \,\epsilon^{2/3}$, which contradicts
\eqref{eq:resol} in the case $m = 1$ because $\epsilon \to 0$ as $\delta \to
0$. Similar arguments can be used to prove the optimality of
Assumption~\ref{Hypv} with $m = 2$ for functions $v$ that have nondegenerate
critical points.

\section{Energy estimates}\label{sec3}

In this section we give an alternative proof of a slight variation of Theorems
\ref{thm:main2}--\ref{thm:main3}, using a direct energy method usually referred
to as hypocoercivity \cite{Vi}.To avoid a few technicalities related to boundaries, 
we restrict ourselves to the following two cases:
\begin{enumerate}
 \item[1.] $\Omega = \TT^d$, the $d$-dimensional periodic box;
 \item[2.] $\Omega\subset \RR^d$ a smooth bounded domain, with {\em homogeneous
Dirichlet boundary conditions} for $g$.
\end{enumerate}
The result involves an energy functional of the form
\begin{align}\label{eq:hypofun}
  \Phi=\frac12 \left[\|g\|^2+\alpha\|\nabla g\|^2+2\beta \Re \l ik g\nabla v ,
  \nabla g \r+\gamma  k^2\| g\nabla v \|^2\right],
\end{align}
where $\langle\cdot,\cdot\rangle$ denotes the scalar product in $X = L^2(\Omega)$ 
and $\|\cdot\|$ the associated norm. The parameters $\alpha,\beta,\gamma$ depend on 
$\nu,k$ in a different way according to whether we consider the enhanced dissipation 
regime or the Taylor dispersion regime. We prove the following result:

\begin{theorem}\label{thm:main3.1}
Assume that $v : \overline{\Omega} \to \RR$ is a smooth Morse function with 
no critical point on the boundary $\partial\Omega$. There exist positive 
constants $\beta_0,C_3$ such that, for all $\nu > 0$, all $k \neq 0$, and all 
initial data $g_0 \in H^1_0(\Omega)$, the solution of \eqref{eq:geq} satisfies, 
for all $t \ge 0$,  
\begin{equation}\label{lamnuk3.1}
  \Phi(t) \,\le\, \e^{-C_3 \lambda_{\nu,k}t}\,  \Phi(0) \,, 
  \quad\hbox{where }~ \lambda_{\nu,k} \,=\,
  \begin{cases} \nu^{\frac{m}{m+2}} |k|^{\frac{2}{m+2}} & \hbox{if }~ 
  0 < \nu \le \beta_0 |k|\,,\\ \frac{k^2}{\nu} & \hbox{if }~ 0 < \beta_0|k| \le \nu\,.
  \end{cases}
\end{equation}
Here $m = 1$ if $v$ has no critical point in $\Omega$, and $m = 2$ if $v$ has at least 
one critical point in $\Omega$.
\end{theorem}

\begin{remark}\label{rem:torus}
If $\Omega = \TT^d$, it is understood that $\overline{\Omega} = \TT^d$, $\partial
\Omega = \emptyset$, and $H^1_0(\Omega) = H^1(\TT^d)$. In that case, any Morse 
function $v$ has critical points in $\Omega$, so that we necessarily have $m = 2$. 
\end{remark}

The decay rate $\lambda_{\nu,k}$ in \eqref{lamnuk3.1} is essentially the same as
the one in \eqref{lamnuk}, except for the threshold parameter $\beta_0$ which
will be taken smaller than $1$ in the proof of Theorem~\ref{thm:main3.1}. 
However, the main difference with the arguments developed in Section~\ref{sec2}
is the use of the $H^1$-type energy \eqref{eq:hypofun}. A simple argument allows
us to translate estimate \eqref{lamnuk3.1} on the energy functional $\Phi$ into
a semigroup bound of the form \eqref{lamnuk}, as is stated below.

\begin{corollary}\label{cor:main3.1}
Under the assumptions of Theorem \ref{thm:main3.1}, there exist positive 
constants $\beta_0,C_1, C_2$ such that, for all $\nu > 0$, all $k \neq 0$, and all 
initial data $g_0 \in L^2(\Omega)$, the solution of \eqref{eq:geq} satisfies the
estimate 
\begin{equation}\label{lamnuk.cor}
  \|g(k,t)\| \,\le\, C_1\, \left(1+\frac{|k|}{\nu}\right)^\frac{m-1}{m+2}\,
  \e^{-C_2 \lambda_{\nu,k}t}\,\|g_0\|\,, 
\end{equation}
 for all $t \ge 0$, where $\lambda_{\nu,k}$ is given by \eqref{lamnuk3.1}.
\end{corollary}

\begin{remark}
In the enhanced dissipation regime, when $m\ge 2$, the prefactor appearing in
\eqref{lamnuk.cor} implies a logarithmic correction on the decay rate
$\lambda_{\nu,k}$, as already noticed in \cite{BCZ}. Such a correction is not
present when $m = 1$, namely when there are no critical points. Also, in the
Taylor dispersion case, the prefactor is harmless since
$1 \le 1 + \frac{|k|}{\nu} \le 1 + \beta_0^{-1}$.

In the case of homogeneous Dirichlet boundary conditions, the decay estimate 
\eqref{lamnuk.cor} is not interesting in the Taylor dispersion regime, because
it is in fact weaker than what can be deduced from the simple energy balance 
\eqref{eq:L2balance} in view of the Poincar\'e inequality. 
\end{remark}

The rest of this section is devoted to proving Theorem~\ref{thm:main3.1}. 
 
\subsection{Energy identities}
We start the discussion with some energy identities that will be used to build
the hypocoercivity functional $\Phi$ in \eqref{eq:hypofun}.  Below, we indicate
by $\Delta$ and $\nabla$ the Laplacian and the gradient with respect to the space 
variable $y \in \Omega$. 

\begin{lemma}\label{lem:enest}
Let $g$ solve \eqref{eq:geq} either in the torus $\TT^d$, or in a bounded 
domain $\Omega \subset \RR^d$ with homogeneous Dirichlet boundary conditions. 
Then we have the following balances:
 \begin{align}
  \frac12\ddt \|g\|^2 +\nu\|\nabla g\|^2 \,&=\, 0\,,\label{eq:L2balance} \\
  \frac12\ddt \|\nabla g\|^2+\nu\|\Delta g\|^2 \,&=\, -\Re \l i kg\nabla v,\nabla g \r\,,
  \label{eq:H1balance} \\
  \ddt \Re \l ikg \nabla v ,\nabla g \r+k^2 \| g\nabla v \|^2 \,&=\, -2\nu \Re \l ik 
  \nabla v \cdot\nabla g,\Delta g\r-\nu \Re \l ikg \Delta v ,\Delta g\r\,,
  \label{eq:crossterm}\\
  \frac12\ddt \| g \nabla v \|^2+\nu\|  |\nabla v| \nabla g\|^2 \,&=\,
  -2 \nu  \Re\l g D^2 v \nabla v , \nabla g \r\,.\label{eq:gamma-weight} 
\end{align}
\end{lemma}

\begin{proof}
All the identities are established by direct computation, using integration by 
parts. There are no boundary terms if $\Omega = \TT^d$, and in the other case the
contributions from the boundary vanish thanks to the homogeneous Dirichlet
conditions. The $L^2$ balance \eqref{eq:L2balance} follows directly by testing
\eqref{eq:geq} with $g$ and using the antisymmetry property $\Re\l i v g,g\r=0$.
Testing \eqref{eq:geq} with $-\Delta g$ we also obtain \eqref{eq:H1balance} by a
simple integration by parts. Turning to \eqref{eq:crossterm}, we use \eqref{eq:geq} 
to compute
\[
  \ddt \Re \l ik g \nabla v ,\nabla g \r
  \,=\,\nu \Re \left[\l i k (\nabla v) \Delta g,\nabla g \r+\l i kg \nabla v ,\nabla 
  \Delta g \r\right] + k^2\Re \left[\l  vg\nabla v,\nabla g \r-\l g\nabla v ,\nabla
  (vg )\r\right].
\]
We treat the $\nu$ term integrating by parts as
\[
  \Re \left[\l i  (\nabla v) \Delta g,\nabla g \r+\l i g \nabla v ,\nabla \Delta g 
  \r\right] \,=\,-2 \Re \l i\nabla v \cdot \nabla g,\Delta g\r-
  \Re\l i g\Delta v , \Delta g \r\,,
\]
while for the second term we compute
\begin{align*}
  \l vg \nabla v ,\nabla g \r-\l g \nabla v ,\nabla  (vg )\r\,=\,-\| g \nabla v \|^2\,,
\end{align*}
and \eqref{eq:crossterm} follows.
For \eqref{eq:gamma-weight}, thanks to antisymmetry we have
\begin{align*}
  \frac12\ddt\| g \nabla v \|^2
  \,&=\, \nu \Re\l g \nabla v, (\nabla v) \Delta g \r - \Re\l g\nabla v, ik v  g \nabla v  
  \r\notag\\ \,&=\,- \nu\| |\nabla v| \nabla g\|^2-2 \nu  \Re\l g D^2 v \nabla v,  
  \nabla g \r\,.
\end{align*}
This concludes the proof. 
\end{proof}

\begin{remark}\label{rem:Neumann}
The relations \eqref{eq:L2balance}, \eqref{eq:H1balance}, and \eqref{eq:gamma-weight} 
remain valid if $g$ satisfies the homogeneous Neumann conditions on $\partial\Omega$, 
but to obtain \eqref{eq:crossterm} one has to assume in addition that the normal 
derivative of the shear velocity $v$ vanishes identically on the boundary. This 
additional hypothesis is not very natural, as it is not satisfied in many classical
examples, such as the cylindrical Poiseuille flow. Moreover, it conflicts with 
our forthcoming assumption that $v$ has no critical points on $\partial\Omega$, 
see Proposition~\ref{prop:spect} below. For these reasons, we prefer assuming in 
this section that $g = 0$ on $\partial\Omega$, or alternatively that $\Omega
 = \TT^d$. 
\end{remark}

For each $k\in\RR$, the (frequency-localized) energy functional $\Phi$ in
\eqref{eq:hypofun} depends on the positive coefficients $\alpha,\beta,\gamma$,
to be chosen depending on $k$ and $\nu$. For the moment, we assume that
\begin{align}\label{eq:constraint1}
  \frac{\beta^2}{\alpha\gamma} \,\le\, \frac{1}{16}\,,
\end{align}
a condition that guarantees the coercivity of $\Phi$. Indeed, since
\[
  2\beta |k||\l g\nabla v ,\nabla g \r| \le 2\beta |k|\|g \nabla v \| \|\nabla g \| \,\le\,
 \frac{\alpha}{4} \|\nabla g \|^2 + \frac{4\beta^2 k^2}{\alpha} \| g\nabla v \|^2 \,\le\,
 \frac{\alpha}{4} \|\nabla g \|^2 + \frac{\gamma k^2}{4} \| g\nabla v \|^2\,,
\]
we obtain
\begin{align}\label{eq:coercivity}
  \frac18 \left[4\|g\|^2+3\alpha\|\nabla g\|^2+3\gamma k^2\|g \nabla v  \|^2\right]
  \,\le\,\Phi\,\le\,\frac18 \left[4\|g\|^2+5\alpha\|\nabla g\|^2+5\gamma k^2\|g \nabla v 
  \|^2\right]\,.
\end{align}
Moreover, Lemma \ref{lem:enest} readily implies that
\begin{align}\label{eq:energy}
  \ddt \Phi+\nu\|\nabla g\|^2 &+\alpha\nu\| \Delta g\|^2+\beta k^2 \| g \nabla v \|^2
  +\gamma\nu k^2\| |\nabla v| \nabla g\|^2\notag\\
  \,&=\, -\alpha \Re\l ik g \nabla v ,\nabla g \r-2\beta\nu \Re \l ik\nabla v \cdot 
  \nabla g,\Delta g\r \\ &\quad\,\, -\beta\nu \Re\l ik g \Delta v ,\Delta g\r-2\gamma\nu k^2 
  \l g D^2 v \nabla v  ,  \nabla g \r\,.\notag
\end{align}
In addition to \eqref{eq:constraint1}, let us assume from now on that 
\begin{align}\label{eq:constraint2}
  \frac{\alpha^2}{\nu} \,\le\,\beta\,.
\end{align}
In this way, the first term on the right-hand side of \eqref{eq:energy} can be estimated as
\begin{equation}\label{aux1}
  \alpha |\l ikg \nabla v ,\nabla g \r|\,\le\, \alpha |k|\| g \nabla v  \| \|\nabla g\|
  \,\le\,  \frac{\nu}{2}\|\nabla g\|^2+  \frac{\alpha^2 k^2}{2\nu} \| g\nabla v \|^2\,\le\,
  \frac{\nu}{2}\|\nabla g\|^2+  \frac{\beta k^2}{2} \| g \nabla v  \|^2\,.
\end{equation}
Moreover, thanks to  \eqref{eq:constraint1} we have
\begin{equation}\label{aux2}
  2\beta\nu |\l ik\nabla v \cdot\nabla g,\Delta g\r| \,\le\, 2\beta\nu |k|\| |\nabla v| 
  \nabla g\| \| \Delta g\| \,\le\, \frac{\alpha \nu}{4}\| \Delta g\|^2+ \frac{\gamma \nu 
  k^2}{4}\|  |\nabla v| \nabla g\|^2\,,
\end{equation}
and 
\begin{equation}\label{aux3}
  \beta\nu|\l ikg \Delta v,\Delta g\r| \,\le\, \beta\nu |k| \| g \Delta v\| \| \Delta g\|
  \,\le\, \frac{\alpha \nu}{4}\| \Delta g\|^2+ \frac{\gamma \nu k^2}{16}\| g\Delta v\|^2.
\end{equation}
Finally, we have
\begin{equation}\label{aux4}
  2\gamma\nu k^2 |\l g D^2 v \nabla v  ,  \nabla g \r|\,\le\, \frac{\gamma \nu k^2}{4}\| 
  |\nabla v| \nabla g\|^2+4\gamma \nu k^2\|g D^2 v \|^2\,.
\end{equation}
Combining \eqref{eq:energy} and \eqref{aux1}--\eqref{aux4}, we arrive at
\begin{equation}\label{eq:energyest1}
  \ddt \Phi+\frac{\nu}{2}\|\nabla g\|^2+\frac{\alpha\nu}{2}\| \Delta g\|^2+\frac{\beta 
  k^2}{2} \| g \nabla v \|^2+\frac{\gamma\nu k^2}{2}\|  |\nabla v| \nabla g\|^2
  \,\le\, 5 \gamma\nu k^2\| g D^2v  \|^2.
\end{equation}
In order to obtain a differential inequality for the functional $\Phi$, it remains to
bound the remainder term in the right-hand side of \eqref{eq:energyest1}, and to 
show that $\Phi$ itself can be controlled using the positive terms in the left-hand side. 

\subsection{A semi-classical estimate}
One of the key elements of the proof via hypocoercivity is the following
inequality, which we state in a fairly general way.

\begin{proposition}\label{prop:spect}
Let $w:\Omega\to \RR$ be a smooth function such that 
\begin{enumerate} [label=(H\arabic*), ref=(H\arabic*)]
\item \label{hyp1} $w\in C^1(\overline{\Omega})$, and $\nabla w \neq 0$ on $\de\Omega$;
\item  \label{hyp2} $w$ has a finite number of critical points $y_1,\ldots, y_N$ in 
$\Omega$. For $1\le j\le N$ there exist a radius $R_j\in(0,1)$, an integer $m_j\ge 2$, 
and a constant $C \ge 1$ such that
\begin{align}\label{nablawest}
  \frac{1}{C}|y-y_j|^{m_j-1}\,\le\, |\nabla w(y)|\,\le\, C |y-y_j|^{m_j-1}, \qquad \forall\,
  y\in B(y_j, R_j)\,,
\end{align}
where $B(y_j, R_j)$ denotes the ball of radius $R_j$ centered at $y_j$. 
\end{enumerate}
Then there exists a constant $C_{sp}\ge 1$ such that, for all $\sigma\in (0,1]$ and 
all $\varphi\in H^1(\Omega)$, the following inequality holds
\begin{align}\label{eq:spectralgap3}
  \sigma^{\frac{m-1}{m}}\|\varphi\|^2\,\le\, C_{sp}\left[\sigma\|\nabla \varphi\|^2 +\| \varphi 
  \nabla w \|^2\right]\,,
\end{align}
where $m = \max\{m_1,\ldots,m_N\}$ if $N \ge 1$ and $m=1$ if $w$ has no critical 
point in $\Omega$. 
\end{proposition}

Estimate \eqref{eq:spectralgap3} means that the ground state of the Schr\"odinger 
operator $-\sigma\Delta + |\nabla w|^2$ in $L^2(\Omega)$ is bounded from below by
$C_{sp}^{-1}\,\sigma^{\frac{m-1}{m}}$ in the semi-classical limit $\sigma \to 0$. 
There is of course an abundant literature on spectral bounds for semi-classical
Schr\"odinger operators. The background material can be found in the excellent
references \cites{HS1,HN}, where estimate \eqref{eq:spectralgap3} is established
at least in some special cases. For the sake of completeness, we give below a short 
proof of \eqref{eq:spectralgap3} under our general assumptions. 

\medskip
The proof of Proposition \ref{prop:spect} relies on the following lemma.

\begin{lemma}\label{lem:radialspect}
Let $R > 0$ and denote by $B_R\subset \RR^d$ the ball of radius $R$ centered at $0$. 
If $\varphi\in H^1_0(B_R)$ and $\ell\in \NN$, then 
\begin{align}\label{eq:spectralgap2}
  \| \varphi\|_{L^2(B_R)} \,\le\, C\| \nabla \varphi\|^{\frac{\ell}{\ell+1}}_{L^2(B_R)} \| 
  |y|^\ell \varphi\|^\frac{1}{\ell +1}_{L^2(B_R)}\,,
\end{align}
where the constant $C > 0$ depends only on $d$ and $\ell$.
\end{lemma}

\begin{proof}
There is nothing to prove if $\ell = 0$, so we assume henceforth that $\ell \ge 1$. 
Passing to polar coordinates and decomposing $\varphi$ in spherical harmonics 
as in the proof of Lemma~\ref{lem:annulus} below, we see that it is sufficient 
to prove \eqref{eq:spectralgap2} in the particular case where $\varphi$ is
{\em radially symmetric}. Under that assumption, we can integrate by parts 
and obtain
\[
  \| \varphi\|^2 \,=\, A_d\int_0^R \varphi(r)^2 r^{d-1}\dd r \,=\, -\frac{2 A_d}{d}
  \int_0^R r\varphi(r)\varphi'(r) r^{d-1} \dd r\,,
\]
where $A_d = 2\pi^{d/2}\Gamma(d/2)^{-1}$ is the area of the unit sphere in 
$\RR^d$. By H\"older's inequality we then find
\[
  A_d \int_0^R r|\varphi(r)||\varphi'(r)| r^{d-1} \dd r \,\le\,  \|\varphi'\|
 \|r^\ell\varphi\|^\frac{1}{\ell} \|\varphi\|^{1-\frac1\ell}\,,
\]
hence
\[
  \| \varphi\|^2 \,\le\, \frac{2}{d}\|\varphi'\| \| r^\ell \varphi\|^\frac{1}{\ell} 
  \|\varphi\|^{1-\frac1\ell}\,.
\]
This gives the desired inequality \eqref{eq:spectralgap2} for a radially symmetric
$\varphi$, and the general case follows. 
\end{proof}

We are now ready to prove the semi-classical estimate \eqref{eq:spectralgap3}.

\begin{proof}[Proof of Proposition \ref{prop:spect}]
For definiteness we consider the case of a bounded domain $\Omega \subset \RR^d$; 
the argument is similar in the periodic case.  In Assumption~\ref{hyp2}, we
suppose without loss of generality that the balls $B(y_j,R_j)$ are pairwise
disjoint. For each $j\in\{1,\ldots,N\}$, let $\chi_j$ be a smooth cut-off
function such that $\supp(\chi_j) \subset B(y_j,R_j)$ and $\chi_j = 1$ on
$B(y_j,R_j/2)$. We can also assume that there exists a smooth function $\chi_0$
such that
\[
  \chi_0(y)^2 + \sum_{j=1}^N \chi_j(y)^2 \,=\, 1\,, \qquad \hbox{for all }
  y \in \RR^d\,,
\]
so that the family $\{\chi_j^2\}$ for $j=0,\dots,N$ is a smooth partition of unity. 

Fix $\varphi \in H^1(\Omega)$. For any $j \in \{1,\dots,N\}$ we have 
$\varphi \chi_j\in H^1_0(B(y_j,R_j))$, so using assumption \eqref{nablawest} 
and Lemma~\ref{lem:radialspect} with $\ell = m_j-1$, we obtain
\[
  \| \varphi \chi_j \|_{L^2(B(y_j,R_j))} \,\le\, C \| \nabla (\varphi \chi_j)\|^{
  \frac{m_j-1}{m_j}}_{L^2(B(y_j,R_j))} \|  \varphi \chi_j\nabla w\|^\frac{1}{m_j}_{
  L^2(B(y_j,R_j))}\,, \qquad j \ge 1\,.
\]
Equivalently, by Young's inequality, there exists a constant $c_1 > 0$ such that
\begin{equation}\label{phichi}
  \sigma^{\frac{m_j-1}{m_j}}\|\varphi \chi_j \|^2\,\le\, c_1\left[\sigma\|\nabla (\varphi 
  \chi_j )\|^2 +\|  \varphi \chi_j \nabla w\|^2\right]\,, \qquad j\ge 1\,.
\end{equation}
We remark that inequality \eqref{phichi} is also satisfied when $j = 0$, with 
$m_0 = 1$, since by construction the function $|\nabla w|$ is bounded away from 
zero on the support of $\varphi \chi_0$. Taking this observation into account, 
we can compute
\begin{align*}
  \sigma\|\nabla \varphi \|^2 +\| \varphi\nabla w  \|^2
  \,&=\,\sum_{j\ge 0} \sigma\| \chi_j \nabla \varphi \|^2 +\| \varphi \chi_j \nabla w \|^2\\
  \,&=\,\sum_{j\ge 0} \sigma\|  \nabla (\varphi \chi_j)  -\varphi \nabla \chi_j\|^2 +\| 
  \varphi \chi_j  \nabla w \|^2\\
  \,&=\, \sum_{j\ge 0} \sigma\|  \nabla (\varphi \chi_j) \|^2 +\| \varphi \chi_j  
  \nabla w\|^2 -2\sigma \l  \nabla (\varphi \chi_j),\varphi \nabla \chi_j\r+\sigma
  \|\varphi \nabla \chi_j\|^2\\
  \,&\ge\, \frac12\sum_{j\ge 0} \left[\sigma\|  \nabla (\varphi \chi_j) \|^2 +\| \varphi 
  \chi_j\nabla w  \|^2\right] -c_2\sigma\|\varphi \|^2\\
  \,&\ge\, \frac{1}{2c_1}  \sum_{j\ge 0}\sigma^{\frac{m_j-1}{m_j}} \|\varphi \chi_j \|^2-c_2
  \sigma\|\varphi \|^2\\
  \,&=\, \frac{1}{2c_1}  \sigma^{\frac{m-1}{m}} \|\varphi \|^2-c_2\sigma\|\varphi \|^2
  \,\ge\, \frac{1}{4c_1}  \sigma^{\frac{m-1}{m}} \|\varphi \|^2\,,
\end{align*}
provided $\sigma\in (0,\sigma_0]$ for some $\sigma_0$ small enough. A simple
rescaling of $\sigma$ then implies \eqref{eq:spectralgap3}, hence concluding the
proof.
\end{proof}

\subsection{Velocity profiles with simple critical points}
If $v$ is a Morse function, namely if all critical points of $v$ are
nondegenerate, the strategy is to estimate the right-hand side of
\eqref{eq:energyest1} using the smoothness of $v$, and then to apply inequality
\eqref{eq:spectralgap3} with $m=2$ (or $m = 1$ if $v$ has no critical
point). For the sake of clarity, we concentrate here on the harder case $m = 2$.

\begin{proof}[Proof of Theorem \ref{thm:main3.1}, case $m=2$]
Our starting point is inequality \eqref{eq:energyest1}. Since $v\in C^2(
\overline{\Omega})$, there exists a positive constant $c_v = c(\|v\|_{W^{2,\infty}})$ 
such that $\|D^2 v\|_{L^\infty} \le c_v/5$, hence 
\[
  \ddt \Phi+\frac{\nu}{2}\|\nabla g\|^2+\frac{\beta k^2}{2} \| g\nabla v \|^2
  \,\le\, c_v \gamma\nu k^2\|  g\|^2\,.
\]
We need to show that the right-hand side can be absorbed in the left-hand side
thanks to inequality \eqref{eq:spectralgap3} and to a suitable choice of the
parameters $\alpha,\beta,\gamma$, in compliance with \eqref{eq:constraint1} and
\eqref{eq:constraint2}. Assuming for the moment that such a choice can be made,
so that
\begin{equation}\label{eq:spectralgap12}
  c_v \gamma\nu k^2\| g\|^2\,\le\, \frac{\nu}{4}\|\nabla g\|^2+\frac{\beta k^2}{4}
  \|g\nabla v \|^2\,,
\end{equation}
we then find
\[
  \ddt \Phi+\frac{\nu}{4}\|\nabla g\|^2+\frac{\beta k^2}{4}\|g \nabla v \|^2\,\le\, 0\,.
\]
Using \eqref{eq:spectralgap12} once more, we obtain
\[
  \ddt \Phi+\frac{c_v \gamma\nu k^2}{2}\| g\|^2+\frac{\nu}{8}\|\nabla g\|^2+\frac{
  \beta k^2}{8}\|g\nabla v \|^2\,\le\, 0\,,
\]
or equivalently
\begin{equation}\label{eq:energyest3}
  \ddt \Phi+\frac{c_v \gamma\nu k^2}{8} \left[4\|g\|^2+\frac{1}{5c_v\alpha \gamma k^2} 
  5\alpha\|\nabla g\|^2+\frac{\beta}{5c_v \gamma^2\nu k^2} 5\gamma k^2\|g\nabla v \|^2
  \right]\,\le\, 0\,.
\end{equation}
In order to fulfill \eqref{eq:constraint1} and \eqref{eq:constraint2} with an
equality, we make the choices
\begin{equation}\label{eq:choice}
  \alpha^2 \,=\, \beta \nu\,,\qquad \gamma \,=\,\frac{16\beta^{3/2}}{\nu^{1/2}}\,,
\end{equation}
and rewrite \eqref{eq:energyest3} as
\begin{equation}\label{eq:energyest4}
  \ddt \Phi+2 c_v \beta^{3/2} \nu^{1/2} k^2 \left[4\|g\|^2+\frac{1}{80 c_v\beta^2 k^2} 
  5\alpha\|\nabla g\|^2+\frac{1}{1280 c_v \beta^2 k^2} 5\gamma k^2\|g\nabla v \|^2
  \right]\,\le\, 0\,.
\end{equation}
We now consider two complementary regimes, verify \eqref{eq:spectralgap12} and
close a proper Gronwall estimate for $\Phi$.

\medskip

\noindent $\diamond$ \emph{Enhanced dissipation}. For some $\beta_0\in (0,1)$ to be fixed 
and independent of $\nu,k$, we take
\begin{equation}\label{eq:beta}
  \beta \,=\, \frac{\beta_0}{|k|}\,, \quad \hbox{and we assume}\quad
  \frac{\nu}{ |k|}\,\le\, \beta_0\,.
\end{equation}
To verify \eqref{eq:spectralgap12}, we use inequality \eqref{eq:spectralgap3} with 
$w = v$, $m = 2$, and
\[
  \sigma \,=\, \frac{\nu}{\beta k^2}\,\le\, 1\,.
\]
We thus obtain
\begin{equation}\label{eq:inequality1}
  \beta^{1/2}\nu^{1/2}|k|\|g\|^2\,\le\, C_{sp}\left[\nu\|\nabla g \|^2 +\beta k^2 
  \| g\nabla v  \|^2\right]\,.
\end{equation}
It follows that inequality \eqref{eq:spectralgap12} is verified provided
\[
  c_v \gamma\nu k^2 \,\le\, \frac{\beta^{1/2}\nu^{1/2}|k|}{4C_{sp}}\,.
\]
In view of \eqref{eq:choice} and \eqref{eq:beta}, this is equivalent to
$\beta_0 \le (64\,c_v C_{sp})^{-1}$, which is simply requiring $\beta_0$ to be 
small enough. Going back to \eqref{eq:energyest4} and possibly reducing $\beta_0$ 
further so that $1280\,c_v \beta_0^2 \le 1$, we find from the coercivity of $\Phi$ 
in \eqref{eq:coercivity} that
\begin{align}\label{eq:finalphi1}
  \ddt \Phi + 16\,c_v \beta_0^{3/2}  \nu^{1/2}|k|^{1/2}\Phi\,\le\, 0\,,
\end{align}
which immediately implies \eqref{lamnuk3.1}. 

\medskip

\noindent $\diamond$ \emph{Taylor dispersion}.
If $\nu |k|^{-1}\ge \beta_0$, we take
\begin{align}\label{eq:beta:bis}
  \beta \,=\, \frac{\beta_1}{\nu}\,,
\end{align}
for some $\beta_1\in (0,\beta_0^2]$ to be fixed and independent of $\nu,k$.
Now, choosing $\sigma=1$ in \eqref{eq:spectralgap3} and using the assumption that 
$\beta_1\le \beta_0^2$, we find
\[
  \beta k^2\| g\|^2\,\le\, C_{sp}\left[\nu\|\nabla g\|^2+ \beta k^2\|g\nabla v \|^2
  \right]\,.
\]
It follows that inequality \eqref{eq:spectralgap12} is verified provided
\[
  c_v \gamma\nu k^2 \,\le\, \frac{\beta k^2}{4C_{sp}}\,.
\]
From \eqref{eq:choice} and \eqref{eq:beta:bis}, this is equivalent to
$\beta_1^{1/2} \le (64\,c_v C_{sp})^{-1}$. Going back to \eqref{eq:energyest4}, we 
possibly reduce $\beta_1$ further so that $1280\,c_v \beta_1^2 \le \beta_0^2$.  
From the coercivity of $\Phi$ in \eqref{eq:coercivity}, we then  find 
\[
  \ddt \Phi + 16\,c_v \beta_1^{3/2}\frac{k^2}{\nu}\Phi\,\le\, 0\,.
\]
and the proof is now complete. 
\end{proof}

\begin{remark}
In the simpler case $m=1$, the only difference is the scaling of
$\alpha,\beta,\gamma$ with respect to $\nu,k$, in the enhanced dissipation
regime.  Specifically, the choice of $\beta$ in \eqref{eq:beta} has to be
changed to
\begin{align}\label{eq:beta:tris}
  \beta=\beta_0\frac{\nu^{1/3}}{|k|^{4/3}}\,.
\end{align}
In fact, an even simpler proof can be carried out, without the use of the term
multiplied by $\gamma$ in \eqref{eq:hypofun}. See \cite{CZ20} for a proof in the
one dimensional case.
\end{remark}

\begin{remark}\label{rem:TayHypo}
In the Taylor dispersion regime, we really only used that $v$ is twice
continuously differentiable and not identically constant. Indeed, all
we need is that inequality \eqref{eq:spectralgap3} holds when $\sigma=1$, 
and this does not require any particular structure. Therefore, the above 
argument also gives an alternative proof of Theorem~\ref{thm:main3} under 
slightly more restrictive regularity assumptions on $v$.
\end{remark}

It remains to give a proof of the semigroup estimate \eqref{lamnuk.cor}. The 
argument follows essentially that of \cites{CZEW,CZP}, and we give the details 
here for completeness.

\begin{proof}[Proof of  Corollary \ref{cor:main3.1}.]
We first consider the enhanced dissipation regime. Let $T_{\nu,k} = 1/\lambda_{\nu,k}$ be 
the relevant time scale. From the energy balance \eqref{eq:L2balance} and the 
mean-value theorem, we may find a time $t_0\in (0,T_{\nu,k})$ such that
\[
  2 \| \nabla g(t_0)\|^2 \,\le\, \frac{\lambda_{\nu,k}}{\nu}\| g_0\|^2 \,=\, 
 \left(\frac{|k|}{\nu}\right)^{\frac{2}{m+2}} \|g_0\|^{2}\,.
\]
In turn, from the definition of $\alpha, \beta,\gamma$ in \eqref{eq:choice} and
\eqref{eq:beta} (or in \eqref{eq:beta:tris} for $m=1$), the above inequality can
be rewritten as
\begin{align}\label{eq:t0esti1}
  \alpha\| \nabla g(t_0)\|^2 \,\le\, \frac{\beta_0^{1/2}}{2} \| g_0\|^2\,.
\end{align}
Since $\nabla v$ is bounded on $\Omega$ and $t\mapsto \|g(t)\|^2$ is decreasing
by \eqref{eq:L2balance}, we infer from \eqref{eq:coercivity} that
\begin{align}\label{eq:t0estimate}
  \Phi(t_0) \,\le\, \frac18 \left[4\|g(t_0)\|^2+5 \alpha \|\nabla g (t_0) \|^2
  + 5\gamma k^2\|g(t_0)\nabla v \|^2\right]\,\le\, K_0(1+\gamma k^2) \| g_0\|^{2}\,,
\end{align}
for some constant $K_0 > 0$ which is independent on $\nu,k$. Hence, for any 
$t\ge T_{\nu,k}$, the differential inequality \eqref{eq:finalphi1} implies that
\begin{align}\label{Phiend}
  \frac{1}{2} \| g(t)\|^2 \,\le\, \Phi(t)\,\le\, \e^{-16 c_v \beta_0^{3/2}  \lambda_{\nu,k}  (t-t_0)}
  \Phi(t_0)\,\le\,  K_0 (1+\gamma k^2)\,\e^{16 c_v \beta_0^{3/2}}    \e^{-16 c_v \beta_0^{3/2}  
  \lambda_{\nu,k}   t} \| g_0\|^{2}\,.
\end{align}
In terms of powers of $\nu$ and $k$, we have that
\[
   \beta \,\sim\, \frac{\nu^\frac{2-m}{m+2}}{|k|^\frac{4}{m+2}}, \qquad  \gamma \,\sim\, 
   \frac{\nu^\frac{2(1-m)}{m+2}}{|k|^\frac{6}{m+2}} \qquad \Longrightarrow \qquad \gamma k^2
  \,\sim\,  \left(\frac{|k|}{\nu}\right)^{\frac{2(m-1)}{m+2}}\,,
\]
so that \eqref{Phiend} gives the semigroup estimate \eqref{lamnuk.cor} for $t\ge T_{\nu,k}$. 
Since inequality \eqref{lamnuk.cor} is trivially satisfied when $t<T_{\nu,k}$, the 
proof is complete in the enhanced dissipation regime. 

In the Taylor dispersion regime, the argument is analogous, and in fact more 
elementary due to the simple scaling of $\alpha, \beta,\gamma$. Indeed, since 
$\alpha\sim 1$ and $\gamma\sim \nu^{-2}$, we only have to replace \eqref{eq:t0esti1} 
and \eqref{eq:t0estimate} by
\begin{align}\label{eq:t0estimate2}
  \alpha\| \nabla g(t_0)\|^2\,\le\, \frac{\beta_1^{1/2}}{2} \left(\frac{k}{\nu}
  \right)^{\!\!2}\|g_0\|^{2}, \qquad \Phi(t_0)\,\le\, K_0\biggl(1+ \left(
  \frac{k}{\nu}\right)^{\!\!2}\,\biggr)\|g_0\|^{2}\,,
\end{align}
respectively. The rest of the argument is exactly the same as before, and leads to 
\eqref{lamnuk.cor}.  
\end{proof}

\section{Conclusions}\label{sec4}

The results of this paper emphasize the link between enhanced dissipation and
Taylor dispersion in parallel shear flows, thereby demonstrating that these
phenomena, which have a common origin, can be analyzed using the same
mathematical tools. In the Taylor dispersion regime, an optimal decay estimate
is obtained under very general assumptions on the shear velocity $v$, see
Theorem~\ref{thm:main3}, but our approach does not give the classical formula
for the effective diffusion constant in the asymptotic regime where
$|k|\ll \nu$. That formula can be established using homogenization theory
\cites{MK,PS}, see also \cite{BCW} for a rigorous proof based on center manifold
theory. In the enhanced dissipation regime, which requires more precise
assumptions, we obtain (to our knowledge) the first general result concerning 
the higher-dimensional case, see Theorem~\ref{thm:main2}. For simplicity, 
we suppose that the shear velocity is a Morse function, but it is clear 
that our methods can be extended to more general situations, some examples
of which are given in Section~\ref{ssec23}. When the cross-section of the 
domain is one-dimensional, we only need to suppose that the critical points
of $v$ are nondegenerate in the sense of \eqref{eq:vder}, and we recover the 
main conclusions of the previous works \cites{BCZ,GGN}. 

A notable feature of our analysis is to provide for our main results two different
proofs, which are based either on $L^2$ resolvent estimates (Section~\ref{sec2})
or on $H^1$ energy estimates (Section~\ref{sec3}).  Both methods have their own
advantages and drawbacks, and it is worth drawing a little summary at this
point.

\smallskip i) The first approach, based on resolvent estimates for the generator
of the linear evolution equation \eqref{eq:geq}, is very general. It can be used
even if the shear velocity is not smooth, see \cite{Wei}, and it is relatively
insensitive to the choice of the boundary conditions. Thanks to the semigroup
bounds recently obtained in \cites{HS2,Wei}, it gives optimal decay estimates
for the solutions of \eqref{eq:geq} in $L^2(\Omega)$, without the logarithmic
corrections originating from the hypocoercivity method \cite{BCZ}. It relies
entirely on standard techniques for the analysis of linear partial differential
operators, and can therefore be applied to higher-order dissipative operators,
involving for instance the bilaplacian. However, when the cross-section of our
domain has dimension $d \ge 2$, Assumption~\ref{Hypv} on the level sets of
the shear velocity $v$ is not easy to verify, and this is why we restrict
ourselves to the relatively simple case of Morse functions.

\smallskip ii) The second approach, inspired from Villani's work on
hypocoercivity \cite{Vi}, has the advantage of dealing directly with the
evolution equation, which makes it potentially applicable to nonlinear problems
as well (see \cite{BH}), although this possibility has not been widely explored so far. It is
based on rather elementary $H^1$ energy estimates, which however impose some
restrictions concerning the boundary conditions. When the shear velocity $v$ is
a Morse function, it essentially relies on the standard semi-classical estimate
\eqref{eq:spectralgap3}, which is certainly easier to prove than \eqref{EEthin}
in the higher-dimensional case. But if $v$ has degenerate critical points, 
the coefficients $\alpha, \beta, \gamma$ in the functional \eqref{eq:hypofun} 
have to be replaced by $y$-dependent functions, which makes the calculations
more complicated. As a final remark, the estimates given by the hypocoercivity
method are naturally expressed in terms of the $H^1$-type functional $\Phi$,
and logarithmic corrections may appear when translating them into ordinary 
$L^2$ estimates for the evolution equation \eqref{eq:geq}.  It is worth
noticing that the introduction of time-dependent coefficients in $\Phi$ can
remove logarithmic losses \cite{CZHGV22,DelZotto21,LZ21,WZ19}. Moreover,
the robustness of the method allows to treat different dissipative operators,
for instance fractional diffusion \cite{LZ21} or more complicated dissipative
operators \cite{BCZD22}.

\appendix

\section{On $H^1$-thin sets}\label{sec:A}

Assumption~\ref{Hypv} in Section~\ref{sec2} is closely related to a notion of
``thinness'' for subsets of $\RR^d$ which seems quite natural, although we were not
able to locate it in the literature. In this section, we introduce this 
notion and discuss a few elementary properties. 

\begin{definition}\label{def:H1thin}
A set $E \subset \RR^d$ is {\em $H^1$-thin} if there exist positive constants 
$C$ and $\delta_0$ such that, for all $\delta \in (0,\delta_0)$ and all 
$g \in H^1(\RR^d)$, the following inequality holds:         
\begin{equation}\label{eq:H1thin}
  \int_{E_\delta} g(x)^2 \dd x \,\le\, \frac12 \int_{\RR^d} g(x)^2 \dd x
  + C \delta^2 \int_{\RR^d} |\nabla g(x)|^2 \dd x\,,
\end{equation}
where $E_\delta = \bigl\{x \in \RR^d \,;\, \dist(x,E) < \delta\bigr\}$. 
\end{definition}

\begin{remark}\label{rem:poincare}
In particular, if we take $g \in H^1_0(E_\delta)$ in \eqref{eq:H1thin}, 
we see that Poincar\'e's inequality holds in $E_\delta$ with constant 
$\sqrt{2C}\delta$, for all sufficiently small $\delta > 0$. It is not 
clear if this property is sufficient to characterize $H^1$-thin sets. 
\end{remark}

We first observe that the factor $1/2$ in \eqref{eq:H1thin} can be replaced
by an arbitrary real number $\kappa \in (0,1)$ without altering the definition. 

\begin{lemma}\label{lem:kappa}
Fix any $\kappa \in (0,1)$. A set $E \subset \RR^d$ is $H^1$-thin if and only 
if there exist positive constants $C$ and $\delta_0$ such that, for all $\delta 
\in (0,\delta_0)$ and all $g \in H^1(\RR^d)$,         
\begin{equation}\label{eq:H1thin2}
  \int_{E_\delta} g(x)^2 \dd x \,\le\, \kappa\int_{\RR^d} g(x)^2 \dd x
  + C \delta^2 \int_{\RR^d} |\nabla g(x)|^2 \dd x\,.
\end{equation}
\end{lemma}

\begin{proof}
Increasing the value of $\kappa$ obviously makes inequality \eqref{eq:H1thin2} 
weaker. To prove Lemma~\ref{lem:kappa}, we have to show that it is possible 
to {\em decrease} the value of $\kappa$ in \eqref{eq:H1thin2}, at the 
expense of modifying the constants $C$ and $\delta_0$. To see that, assume that 
\eqref{eq:H1thin2} holds for some $\kappa \in (0,1)$, and take $g \in H^1(\RR^d)$. 
For any $N \in \NN^*$, we have
\begin{equation}\label{eq:thin1}
  \int_{E_{(N{+}1)\delta}} g(x)^2 \dd x \,\le\, \kappa \int_{\RR^d} g(x)^2 \dd x
  + C (N{+}1)^2\delta^2 \int_{\RR^d} |\nabla g(x)|^2 \dd x\,,
\end{equation}
provided $0 < \delta < \delta_0/(N+1)$. Define $f \in H^1(\RR^d)$ by $f = \chi g$, 
where
\[
  \chi(x) \,=\, \phi\Bigl(\frac{\dist(x,E) - \delta}{N\delta}\Bigr)\,, \qquad
  \phi(t) \,=\, \begin{cases} 1 & \hbox{if } ~t \le 0 \\[-1mm] 1 - t & 
  \hbox{if } ~0 \le t \le 1 \\[-1mm] 0 & \hbox{if } ~t \ge 1\end{cases}\,.
\]
Note that $f$ vanishes outside $E_{(N{+}1)\delta}$, and coincides with
$g$ on $E_\delta$. Applying \eqref{eq:H1thin2} to $f$, we thus find
\begin{equation}\label{eq:thin2}
\begin{split}
  \int_{E_\delta} g(x)^2 \dd x \,&=\, \int_{E_\delta} f(x)^2 \dd x \,\le\, \kappa
  \int_{\RR^d} f(x)^2 \dd x + C \delta^2 \int_{\RR^d} |\nabla f(x)|^2 \dd x \\
  \,&\le\, \kappa \int_{E_{(N{+}1)\delta}} g(x)^2 \dd x + 2C \delta^2 \int_{E_{(N{+}1)\delta}}
  \bigl(|\nabla g|^2 + |\nabla \chi|^2 g^2\bigr)\dd x\,.
\end{split}
\end{equation}
Since $|\nabla \chi| \le 1/(N\delta)$, we deduce from \eqref{eq:thin1}, 
\eqref{eq:thin2} that
\[
  \int_{E_\delta} g(x)^2 \dd x \,\le\, \Bigl(\kappa^2 + \frac{2C}{N^2}\Bigr)
  \int_{\RR^d} g(x)^2 \dd x + C(N)\delta^2 \int_{\RR^d} |\nabla g(x)|^2 \dd x\,,
\]
for some constant $C(N)$ independent of $\delta$. If we take $N$ large enough,
the coefficient in front of the first integral in the right-hand side can be
made smaller than $\kappa' := \kappa(\kappa+1)/2 < \kappa$. Repeating the argument
a finite number of times, we can thus make the coefficient $\kappa$ in \eqref{eq:H1thin2}
as small as we wish.
\end{proof}

It is clear from the definition that, if $E \subset \RR^d$ is $H^1$-thin, then
any subset $F \subset E$ is $H^1$-thin a fortiori. Also, using
Lemma~\ref{lem:kappa}, it is easy to verify that a finite union of $H^1$-thin
sets is $H^1$-thin. Indeed, if $E,F$ are arbitrarily subsets of $\RR^d$ we have,
for all $g \in H^1(\RR^d)$ and all $\delta > 0$, 
\[
   \int_{(E\cup F)_\delta} g^2 \dd x \,=\, \int_{E_\delta\cup F_\delta} g^2 \dd x
   \,\le\, \int_{E_\delta} g^2 \dd x + \int_{F_\delta} g^2 \dd x\,.
\]
If $E,F$ are $H^1$-thin, both integrals in the right-hand side can 
be estimated as in \eqref{eq:H1thin2} with $\kappa = 1/4$, which 
yields inequality \eqref{eq:H1thin} for $E \cup F$. A less immediate
property is stated in the following lemma. 

\begin{lemma}\label{lem:zero}
If $E \subset \RR^d$ is measurable and $H^1$-thin, then $E$ has zero Lebesgue 
measure.
\end{lemma}

\begin{proof}
We can assume without loss of generality that $E$ is bounded, so that
$|E| < \infty$. Given any $\epsilon > 0$ we define $g_\epsilon = \1_E * 
\chi_\epsilon$, where $\1_E$ is the characteristic function of $E$ and
$\chi_\epsilon$ is a standard approximation of unity: 
\[
  \chi_\epsilon(x) \,=\, \frac{1}{\epsilon^d}\,\chi\Bigl(\frac{x}{
  \epsilon}\Bigr)\,, \qquad \chi \in C^\infty_c(\RR^d)\,, \qquad
  \int_{\RR^d} \chi(x)\dd x \,=\, 1\,.
\]
Since $E$ is $H^1$-thin and $g_\epsilon \in H^1(\RR^d)$, inequality 
\eqref{eq:H1thin} shows that, for any small $\delta > 0$, 
\[
  \int_E g_\epsilon(x)^2\dd x \,\le\, \int_{E_\delta} g_\epsilon(x)^2\dd x
  \,\le\, \frac12 \int_{\RR^d} g_\epsilon(x)^2 \dd x
  + C \delta^2 \int_{\RR^d} |\nabla g_\epsilon(x)|^2 \dd x\,.   
\]
Therefore, taking the limit $\delta \to 0$, we obtain 
\begin{equation}\label{eq:thin3}
  \int_E g_\epsilon(x)^2\dd x \,\le\, \frac12 \int_{\RR^d} g_\epsilon(x)^2 \dd x\,.
\end{equation}
Now, in the limit $\epsilon \to 0$, we have $g_\epsilon \to \1_E$ in
$L^2(\RR^d)$, so that both integrals in \eqref{eq:thin3} converge to the same
value $|E|$. We thus obtain the inequality $|E| \le |E|/2$, which implies that
$|E| = 0$.
\end{proof}

The following lemma is useful to construct concrete examples of 
$H^1$-thin sets. 

\begin{lemma}\label{lem:lip}
The graph of any Lipschitz function $h : \RR^{d-1} \to \RR$ is $H^1$-thin in
$\RR^d$. 
\end{lemma}

\begin{proof}
Let $E = \bigl\{(x,h(x)) \in \RR^d\,;\, x \in \RR^{d-1}\bigr\}$ be the graph of 
$h$, and let $M$ be a Lipschitz constant of $h$. We first observe that, for any 
$\delta > 0$, we have the inclusion
\begin{equation}\label{eq:thin4}
  E_\delta \,\subset\, \Gamma_\delta \,:=\, \bigl\{(x,y) \in \RR^d
  \,;\, x \in \RR^{d-1}\,,~|y - h(x)| < N\delta\bigr\}\,,
\end{equation}
where $N = (1+M^2)^{1/2}$. Indeed, for all $x_1,x_2 \in \RR^{d-1}$ and 
all $z \in \RR$, we have
\begin{align*}
  \big|(x_1,h(x_1)+z) - (x_2,h(x_2))\big|^2 \,&=\, |x_1-x_2|^2 + |z + 
  h(x_1) - h(x_2)|^2 \\
  \,&\ge\, |x_1-x_2|^2 + \bigl(|z| - M|x_1-x_2|\bigr)_+^2 \,\ge\, 
  \frac{z^2}{1+M^2}\,,
\end{align*}
where the last inequality is obvious if $|z| \le M|x_1-x_2|$, and can be obtained
by minimizing the function $a \mapsto a^2 + (|z|-Ma)^2$ in the converse case. 
Taking the infimum over $x_2 \in \RR^{d-1}$, we obtain the estimate
\[
  \dist\Bigl((x_1,h(x_1)+z)\,,\,E\Bigr) \,\ge\, \frac{|z|}{\sqrt{1+M^2}}\,, 
  \qquad \forall\,x_1 \in \RR^{d-1}\,, \quad \forall\,z \in \RR\,,
\]
which in turn implies \eqref{eq:thin4}. Now, if $g \in C^1_c(\RR^d)$, we have for 
any $x \in \RR^{d-1}$:
\[
  \int_{h(x)-N\delta}^{h(x)+N\delta} g(x,y)^2\dd y \,\le\, 2N\delta \,\sup_{y \in \RR}
  g(x,y)^2 \,\le\, 2N\delta\,\biggl(\int_\RR g(x,y)^2 \dd y\biggr)^{1/2} 
  \biggl(\int_\RR\partial_y g(x,y)^2 \dd y\biggr)^{1/2}\,,
\]
where we used the bound $\|f\|_{L^\infty}^2 \le \|f\|_{L^2}\|f'\|_{L^2}$ which 
holds for all $f \in H^1(\RR)$. Integrating both sides over $x \in \RR^{d-1}$ 
and using Schwarz's inequality together with \eqref{eq:thin4}, we arrive at
\[
  \int_{E_\delta} g(x)^2 \dd x\dd y \,\le\, \int_{\Gamma_\delta} g(x)^2 \dd x\dd y 
  \,\le\, 2N\delta\,\|g\|_{L^2} \|\nabla g\|_{L^2}\,.
\]
By density, this bound remains valid for all $g \in H^1(\RR^d)$, and 
\eqref{eq:H1thin} then follows from Young's inequality. 
\end{proof}

It is clear from Definition~\ref{def:H1thin} that the family of $H^1$-thin
sets is invariant under the action of the Euclidean group in $\RR^d$. It 
is also easy to verify that $H^1$-thin sets are stable under dilations,
although the upper bound $\delta_0$ on the parameter $\delta$ has to be
replaced by $\lambda \delta_0$ if $E$ is replaced by $\lambda E$ for some
$\lambda > 0$. Combining these observations with Lemma~\ref{lem:lip}, we conclude
that any submanifold $S$ of $\RR^d$ of nonzero codimension is $H^1$-thin. 
More generally, any $m$-rectifiable set $E \subset \RR^d$ with $m \le d-1$ 
is $H^1$-thin. 

\section{Geometric lemmas}\label{sec:B}

In this section we collect some basic estimates for levels sets of 
Morse functions near critical points, which are used in Section~\ref{ssec22}. 
We assume henceforth that the space dimension $d$ is at least $2$. 
Our starting point is:

\begin{lemma}\label{lem:annulus}
For all $g \in H^1(\RR^d)$ and all $R_2 \ge R_1 \ge 0$, we have
\begin{equation}\label{eq:annulus}
  \int_{R_1 \le |x| \le R_2}g(x)^2\dd x \,\le\, 2 (R_2 - R_1)\,\|g\|_{L^2}
  \,\|\nabla g\|_{L^2}\,.
\end{equation}
\end{lemma}

\begin{proof}
We first prove \eqref{eq:annulus} in the particular case where $g \in C^1_c(\RR^d)$ 
and $g$ is {\em radially symmetric}. Under those assumptions, we can integrate
by parts and obtain, for any $r > 0$, 
\[
  -\int_r^\infty 2 g(s) g'(s)s^{d-1}\dd s \,=\, g(r)^2 r^{d-1} + (d{-}1) 
  \int_r^\infty g(s)^2 s^{d-2}\dd s\,. 
\]
Using Schwarz's inequality, we deduce
\[
  g(r)^2 r^{d-1}  + (d{-}1) \int_r^\infty g(s)^2 s^{d-2}\dd s
  \,\le\, 2 \biggl(\int_r^\infty g(s)^2 s^{d-1}\dd s\biggr)^{1/2}
  \biggl(\int_r^\infty g'(s)^2 s^{d-1}\dd s\biggr)^{1/2}\,. 
\]
In particular, we have 
\[
  A_d\,g(r)^2\,r^{d-1} \,\le\, 2\,\|g\|_{L^2}\, \|\nabla g\|_{L^2}\,, 
  \qquad \forall \,r > 0\,,
\]
where $A_d = 2\pi^{d/2}\Gamma(d/2)^{-1}$ is the area of the unit 
sphere $\SS^{d-1} \subset \RR^d$. Integrating both sides over the interval $[R_1,R_2]$, 
we obtain the desired inequality \eqref{eq:annulus}. 

For a general function $g \in C^1_c(\RR^d)$, we introduce polar coordinates
$x = r\omega$ and use the decomposition
\[
  g(r\omega)\,=\, \sum_{n\in\NN} g_n(r)\,Y_n(\omega)\,, \qquad 
  r \in \RR_+\,, \quad \omega \in \SS^{d-1}\,,
\]
where the spherical harmonics $Y_n(\omega)$ are eigenfunctions of the 
Laplace-Beltrami operator on $\SS^{d-1}$, and are normalized so that the family 
$(Y_n)_{n\in\NN}$ is an orthonormal basis of $L^2(\SS^{d-1})$. Using 
Parseval's identity and the previous step, we deduce that
\begin{align*}
  \int_{R_1 \le |x| \le R_2}g(x)^2\dd x \,&=\, \sum_{n\in\NN} \int_{R_1 \le |x| \le R_2}
  g_n(|x|)^2\dd x \,\le\, 2(R_2-R_1)\sum_{n\in\NN} \|g_n\|_{L^2} \|g_n'\|_{L^2} \\*
  \,&\le\, 2(R_2-R_1)\Bigl(\sum_{n\in\NN} \|g_n\|_{L^2}^2\Bigr)^{1/2} 
  \Bigl(\sum_{n\in\NN} \|g'_n\|_{L^2}^2\Bigr)^{1/2} \,\le\, 
   2 (R_2 - R_1)\,\|g\|_{L^2}\,\|\nabla g\|_{L^2}\,,
\end{align*}
because $\|g\|_{L^2}^2 = \sum_{n\in\NN} \|g_n\|_{L^2}^2$ and $\|\nabla g\|_{L^2}^2 \ge 
\sum_{n\in\NN} \|g_n'\|_{L^2}^2$. This proves \eqref{eq:annulus} for all 
$g \in C^1_c(\RR^d)$, and the general case follows by density.
\end{proof}

Now, let $v : \RR^d \to \RR$ be a smooth function and $m \in \NN^*$ a
nonzero integer. In analogy with \eqref{eq:Edef}, \eqref{eq:EEdef}, we define, 
for all $\lambda \in \RR$ and all $\delta > 0$, 
\begin{equation}\label{eq:Edefs}
  E_{\lambda,\delta}^m \,=\, \bigl\{x \in \RR^d\,;\, |v(x)-\lambda| < \delta^m
  \bigl\}\,, \qquad 
  \EE_{\lambda,\delta}^m \,=\, \bigl\{x \in \RR^d\,;\, \dist(x,E_{\lambda,\delta}^m)
  < \delta\bigl\}\,.  
\end{equation}

\begin{lemma}\label{lem:Morse1}
Assume that $v(x) = |x|^2$ for all $x \in \RR^d$. Then for any $\lambda \in \RR$,
any $\delta > 0$, and any $g \in H^1(\RR^d)$, we have
\begin{equation}\label{eq:vmin}
\begin{split}
  \int_{E_{\lambda,\delta}^2} g(x)^2 \dd x \,\le\, 2\sqrt{2}\,\delta\,\|g\|_{L^2}
  \,\|\nabla g\|_{L^2}\,, \quad  \int_{\EE_{\lambda,\delta}^2} g(x)^2 \dd x \,\le\, 
  2(1{+}\sqrt{3})\,\delta\,\|g\|_{L^2}\,\|\nabla g\|_{L^2}\,.
\end{split}
\end{equation}
\end{lemma}

\begin{proof}
Since $v(x) = |x|^2$, the definition \eqref{eq:Edefs} implies that $E_{\lambda,\delta}^2
\subset \{x \in \RR^d\,;\, R_1 \le |x| < R_2\}$ where $R_1 = (\lambda - 
\delta^2)_+^{1/2}$ and $R_2 = (\lambda + \delta^2)_+^{1/2}$. Considering three
cases according to whether $\lambda \le -\delta^2$, $\lambda \in (-\delta^2,\delta^2)$,  
or $\lambda \ge \delta^2$, it is straightforward to verify that $R_2 - R_1 \le 
\sqrt{2}\,\delta$ in all situations, hence \eqref{eq:annulus} gives the first 
inequality in \eqref{eq:vmin}. The same argument applies to $\EE_{\lambda,\delta}^2$ 
if we define
\[
  R_1 \,=\, \bigl((\lambda - \delta^2)_+^{1/2} - \delta\bigr)_+\,, \qquad
  R_2 \,=\, (\lambda + \delta^2)_+^{1/2} + \delta\,. 
\]
Again considering all possible cases, we find that $R_2 - R_1 \le (1{+}\sqrt{3})
\delta$, and the second inequality in \eqref{eq:vmin} follows in the same way. 
\end{proof}

\begin{lemma}\label{lem:Morse2}
Assume that $d = d_1 + d_2$ with $d_1,d_2 \ge 1$, and that $v(x) = |y|^2 - |z|^2$ for
all $x = (y,z) \in \RR^{d_1} \times \RR^{d_2}$. Then for any $\lambda \in \RR$,
any $\delta > 0$, and any $g \in H^1(\RR^d)$, we have
\begin{equation}\label{eq:vsaddle}
\begin{split}
  \int_{E_{\lambda,\delta}^2} g(x)^2 \dd x \,\le\, 2\sqrt{2}\,\delta\,\|g\|_{L^2}
  \,\|\nabla g\|_{L^2}\,, \quad  \int_{\EE_{\lambda,\delta}^2} g(x)^2 \dd x \,\le\, 
  4(2+\sqrt{2})\,\delta\,\|g\|_{L^2}\,\|\nabla g\|_{L^2}\,.
\end{split}
\end{equation}
\end{lemma}

\begin{proof}
We have by definition
\[
  E_{\lambda,\delta}^2 \,=\, \bigl\{(y,z) \in \RR^d\,;\, |z|^2 + \lambda - \delta^2
  < |y|^2 < |z|^2 + \lambda + \delta^2\bigl\}\,. 
\]
It follows that $E_{\lambda,\delta}^2 \subset \{(y,z) ; R_1(z) \le |y| < R_2(z)\}$ 
where
\[
  R_2(z) \,=\, \bigl(|z|^2 + \lambda + \delta^2\bigr)_+^{1/2}\,, \qquad
  R_1(z) \,=\, \bigl(|z|^2 + \lambda - \delta^2\bigr)_+^{1/2}\,.
\]
As before we have $R_2(z) - R_1(z) \le \sqrt{2}\,\delta$. Thus applying 
Lemma~\ref{lem:annulus} and Fubini's theorem, we obtain
\begin{align*}
  \int_{E_{\lambda,\delta}^2} g(y,z)^2\dd y\dd z \,&\le\, 
  \int_{\RR^{d_2}} \biggl(\int_{R_1(z) \le |y| \le R_2(z)} g(y,z)^2\dd y\biggr) \dd z \\
  \,&\le\, 2\sqrt{2}\,\delta 
  \int_{\RR^{d_2}}\biggl(\int_{\RR^{d_1}} g(y,z)^2\dd y\biggr)^{1/2}
  \biggl(\int_{\RR^{d_1}} |\nabla_y g(y,z)|^2\dd y\biggr)^{1/2}\dd z \\
  \,&\le\,  2\sqrt{2}\,\delta\,\|g\|_{L^2}\,\|\nabla g\|_{L^2}\,,
\end{align*}
which is the first inequality in \eqref{eq:vsaddle}. 

The proof of the second inequality is slightly more complicated. If $(y,z) \in 
\EE_{\lambda,\delta}^2$, then by definition there exists $(\tilde y,\tilde z) \in 
E_{\lambda,\delta}^2$ such that $|y-\tilde y|^2 + |z-\tilde z|^2 < \delta^2$. Let
$\mu = |\tilde y|^2 - |\tilde z|^2 \in (\lambda-\delta^2,\lambda+\delta^2)$. 
If $\mu \ge 0$ we have $|\tilde y| = \sqrt{\mu + |\tilde z|^2}$, hence 
\begin{align*}
  \Bigl| |y| - \sqrt{\mu + |z|^2}\Bigr| \,&\le\,\Bigl| |y| - |\tilde y|\Bigr| + 
  \Bigl| |\tilde y| - \sqrt{\mu + |\tilde z|^2}\Bigr| + \Bigl|
  \sqrt{\mu + |\tilde z|^2} - \sqrt{\mu + |z|^2}\Bigr| \\
  \,&\le\, |y - \tilde y| + |z - \tilde z| \,<\, \sqrt{2}\,\delta\,.
\end{align*}
A similar argument shows that $\bigl| |z| - \sqrt{|\mu| + |y|^2}\bigr| <
\sqrt{2}\,\delta$ if $\mu \le 0$. Thus $\EE_{\lambda,\delta}^2 \subset 
F_{\lambda,\delta} \cup G_{\lambda,\delta}$ where
\begin{align*}
  F_{\lambda,\delta} \,&=\, \Bigl\{(y,z)\,;\, \bigl||y| - \sqrt{\mu + |z|^2}\bigr|
  < \sqrt{2}\,\delta \,\hbox{ for some } \mu \ge 0 \hbox{ with } |\mu - \lambda| 
  < \delta^2\Bigr\}\,, \\
  G_{\lambda,\delta} \,&=\, \Bigl\{(y,z)\,;\, \bigl||z| - \sqrt{|\mu| + |y|^2}\bigr|
  < \sqrt{2}\,\delta \,\hbox{ for some } \mu \le 0 \hbox{ with } |\mu - \lambda| 
  < \delta^2\Bigr\}\,.
\end{align*}
We now distinguish three cases. 

\smallskip \noindent{\bf Case 1\thinspace:} $\lambda \ge \delta^2$. Then 
$G_{\lambda,\delta} = \emptyset$ and $F_{\lambda,\delta} \subset \{(y,z); R_1(z) \le |y| 
< R_2(z)\}$ where
\[
  R_1(z) \,=\, \bigl((\lambda - \delta^2 + |z|^2)^{1/2} - \sqrt{2}\,\delta\bigr)_+\,,
  \qquad R_2(z) \,=\, (\lambda + \delta^2 + |z|^2)^{1/2} + \sqrt{2}\,\delta\,.
\]
It is easy to verify that $R_2(z) - R_1(z) \le (2+\sqrt{2})\delta$, hence proceeding
as above we find
\[
  \int_{\EE_{\lambda,\delta}^2} g(x)^2 \dd x \,\le\, \int_{F_{\lambda,\delta}} g(x)^2 \dd x 
  \,\le\, 2(2 +\sqrt{2})\,\delta\,\|g\|_{L^2}\,\|\nabla g\|_{L^2}\,.
\]

\smallskip \noindent{\bf Case 2\thinspace:} $\lambda \le -\delta^2$. Then 
$F_{\lambda,\delta} = \emptyset$ and a similar argument shows that
\[
  \int_{\EE_{\lambda,\delta}^2} g(x)^2 \dd x \,\le\, \int_{G_{\lambda,\delta}} g(x)^2 \dd x 
  \,\le\, 2(2 +\sqrt{2})\,\delta\, \|g\|_{L^2}\,\|\nabla g\|_{L^2}\,.
\]

\smallskip \noindent{\bf Case 3\thinspace:} $-\delta^2 < \lambda < \delta^2$.
Here both sets $F_{\lambda,\delta}, G_{\lambda,\delta}$ are nonempty, and must be 
considered. We first observe that $F_{\lambda,\delta} \subset  \{(y,z); R_1(z) \le |y| 
< R_2(z)\}$ where 
\[
  R_1(z) \,=\, \bigl((\lambda - \delta^2 + |z|^2)_+^{1/2} - \sqrt{2}\,\delta\bigr)_+\,,
  \qquad R_2(z) \,=\, (\lambda + \delta^2 + |z|^2)^{1/2} + \sqrt{2}\,\delta\,.
\]
One verifies that $R_2(z) - R_1(z) \le (2+\sqrt{2})\,\delta$, and it follows that
\[
  \int_{F_{\lambda,\delta}} g(x)^2 \dd x \,\le\, 2(2+\sqrt{2})\,\delta\,
  \|g\|_{L^2}\,\|\nabla g\|_{L^2}\,.
\]
A similar argument gives the same estimate for the integral over $G_{\lambda,\delta}$, 
and since $\EE_{\lambda,\delta}^2 \subset F_{\lambda,\delta} \cup G_{\lambda,\delta}$ we 
arrive at the second estimate in \eqref{eq:vsaddle} in all cases. 
\end{proof}

\bibliographystyle{abbrv}
\bibliography{TayDispBiblio}

\end{document}